\documentclass{amsart}
\usepackage{amsfonts,amsthm,amssymb,amsmath,enumerate,mathabx}
\usepackage[utf8]{inputenc}
\usepackage[dvipsnames]{xcolor}
\usepackage{color}
\usepackage[all]{xy}
\usepackage{scalerel,stackengine}
\usepackage{soul}
\usepackage{todonotes}
\usepackage{hyperref}

\usepackage{changes}
\definechangesauthor[color=orange,name={Aristides Kontogeorgis}]{AK}

\stackMath
\allowdisplaybreaks
\newcommand{\mayorchico}{\genfrac\lceil\rceil{0.5pt}1}
\newcommand{\lf}{\left\lfloor}
\newcommand{\rf}{\right\rfloor}
\newcommand\reallywidehat[1]{%
\savestack{\tmpbox}{\stretchto{%
  \scaleto{%
    \scalerel*[\widthof{\ensuremath{#1}}]{\kern-.6pt\bigwedge\kern-.6pt}%
    {\rule[-\textheight/2]{1ex}{\textheight}}
  }{\textheight}%
}{0.5ex}}%
\stackon[1pt]{#1}{\tmpbox}%
}
\newtheorem{theorem}[subsection]{Theorem}

\newtheorem*{theorem*}{Theorem}
\newtheorem{proposition}[subsection]{Proposition}
\newtheorem*{proposition*}{Proposition}
\newtheorem{lemma}[subsection]{Lemma}
\newtheorem*{lemma*}{Lemma}
\newtheorem{corollary}[subsection]{Corollary}

\theoremstyle{definition}
\newtheorem*{definition}{Definition}
\newtheorem{remark}[subsection]{Remark}

\newcommand*{\bfrac}[2]{\genfrac{[}{]}{0pt}{}{#1}{#2}}

\title{Arithmetic actions on cyclotomic function fields}
\author{Aristides Kontogeorgis}
\author{ Jacob Kenneth  Ward{$\dagger$}}
\thanks{$\dagger$ 19 September 1982 -  23 June 2019}
\date{\today}
\begin{document}
\maketitle
\begin{abstract} We derive the group structure for cyclotomic function fields obtained by applying the Carlitz action for extensions of an initial constant field. The tame and wild structures are isolated to describe the Galois action on differentials. We show that the associated invariant rings are not polynomial. 
\end{abstract}

\section{Introduction}

Let $\mathbb{F}_q$ denote the finite field of $q$ elements, where $q = p^r$ is a power of the prime integer $p$. Let $d$ be a nonnegative integer. Consider the rational function field $\mathbb{F}_{q^d}(T)$. The Carlitz action for $\mathbb{F}_{q^d}$ is defined as
$$
C_{q^d}(T)(u) = Tu + u^{q^d},
\quad u\in \overline{\mathbb{F}_q(T)}.
$$ 
This defines an action by $T$, which may be extended $\mathbb{F}_{q^d}$-linearly to an action by elements of $\mathbb{F}_{q^d}[T]$, according to 
\begin{equation}\label{Carlitzaction} 
M = \sum a_i T^i\in \mathbb{F}_{q^d}[T], \qquad 
C_{q^d}(M)(u) = M {*}_du=\sum a_i C_{q^d}^i(T)(u). 
\end{equation} The Carlitz $M$-torsion points $C_{q^d}[M]$ are then the torsion $\mathbb{F}_{q^d}$-modules within the algebraic closure $\overline{\mathbb{F}_{q^d}(T)}$ via this action. We note that this does depend upon the choice of $d$. Henceforth, for a polynomial $M \in \mathbb{F}_{q^d}[T]$, we let $K_{q^d,M}$ denote the \emph{cyclotomic function field} for $\mathbb{F}_{q^d}$ and $M$, which is obtained by adjoining to $\mathbb{F}_{q^d}(T)$ the Carlitz $M$-torsion points $C_{q^d}[M]$. These function fields are called \emph{cyclotomic} as they are derived from an exponential function in positive characteristic and enjoy many of the same properties of the classical cyclotomic extensions of $\mathbb{Q}$ \cite[Chapter 12]{salvador2007topics}. 

We suppose henceforth that $M \in \mathbb{F}_q[T]$ and that $M$ splits over $\mathbb{F}_{q^d}[T]$. As in \cite{Ward}, one may determine explicitly the holomorphic differentials for $K_{q^d,M}$. Here, we relate this situation to the original cyclotomic function field $K_{q,M}$. We first prove the essential result:

\begin{theorem} \label{inclusion} $K_{q,M} \subset K_{q^d,M}$.  \end{theorem} 

This allows us to understand the differentials of $K_{q,M}$ in terms of those of $K_{q^d,M}$ by examining Galois covers. Henceforth, we denote $$H_{q^d,M}:=\mathrm{Gal}(K_{q^d,M}/\mathbb{F}_{q^d} K_{q,M}).$$ One then obtains naturally the following tower of fields:
\begin{equation}
\label{fig:tower}
\xymatrix{
 & K_{q^d,M} \ar@{-}[dd] \ar@{-}[dl]_{H_{q^d,M}} \\
 \mathbb{F}_{q^d}K_{q,M}  \ar@{-}[d] \ar@{-}[dr] & \\
 K_{q,M} \ar@{-}[dr] & \mathbb{F}_{q^d}(T) \ar@{-}[d] \\
 & \mathbb{F}_q(T)
}
\end{equation}
The group $H_{q^d,M}$ was studied by Chapman \cite{Chapman} when $M  \in \mathbb{F}_q[T]$ is square-free, in order to give a normal integral basis of the ring of integers of $K_{q,M}$ over $\mathbb{F}_q[T]$. This cannot be done in the same fashion if $M$ contains a square $P^2$ of an irreducible polynomial $P \in \mathbb{F}_q[T]$, as under the wild ramification at $P$ in $K_{q,M}$ the integer ring is no longer a free $\mathbb{F}_q[T]$ module. Here, we study $H_{q^d,M}$ more generally, as the genus is invariant under constant extensions, so that identification of the space of differentials of $K_{q,M}$ reduces to identifying those differentials of $K_{q^d,M}$ which are fixed by the action of $H_{q^d,M}$.

The Frobenius map on $\mathbb{F}_{q^d}/\mathbb{F}_q$ permutes the roots of $M$, and the group $H_{q^d,M}$ is nontrivial: Even in the case where $M = P$ is irreducible over $\mathbb{F}_q[T]$ and $\deg(M) = d$, $|H_{q^d,M}|=(q^d-1)^{d-1}$. The group $H_{q^d,M}$ may be described in a simple, explicit way in terms of the Carlitz action. The abelian Galois group $\mathrm{Gal}(K_{q^d,M}/\mathbb{F}_{q^d}(T))$ is naturally a $\mathrm{Gal}(\mathbb{F}_{q^d}/\mathbb{F}_{q})$-module, and denoting $\mathrm{Gal}(\mathbb{F}_{q^d}/\mathbb{F}_{q})=\langle \sigma\rangle$, the group $H_{q^d,M}$ may be described as 
\[
H_{q^d,M}=
(\sigma-1)\mathrm{Gal}(K_{q^d,M}/\mathbb{F}_{q^d}(T)). 
\] 
This is proven in Lemma \ref{Kummerfun}. In general, the absolute Galois group $\hat{\mathbb{Z}}=\mathrm{Gal}(\bar{\mathbb{F}}_q/\mathbb{F}_q)$ acts on the cyclotomic function fields, and this action may be precisely described using the group structure of $H_{q^d,M}$. This point of view, where the {\em arithmetic} part of a Galois group is acting on the {\em geometric} part, is a unifying notion. One may see, for example, the seminal article of Y. Ihara \cite{IharaProfinite}.

We devote Section 2 to the description of $H_{q^d,M}$. In Section 3, we examine the tame structures arising within this group and give an explicit Kummer generator of the tame part of $H_{q^d,M}$. Section 4 describes the wild component of $H_{q^d,M}$, including the higher ramification groups and different. Cyclotomic function fields may also be viewed as towers of Kummer and Artin-Schreier extensions, whose form we give in Section 5. Section 6 describes the Galois module structure of the differentials of $K_{q,M}$. We concern ourselves with the group $H_{q^d,M}$ as its invariants yield differentials on the constant extension $\mathbb{F}_{q^d} K_{q,M}$ of $K_{q,M}$, and constant extensions do not alter the genus, so that it suffices to give a description of the differentials of $\mathbb{F}_{q^d} K_{q,M}$ in order to understand the Galois module structures for $K_{q,M}$. The holomorphic differentials are comprised of products of generators of the Carlitz torsion modules. In order to compute the $H_{q^d,M}$-invariant differentials, we will employ modular invariant theory. For fixed $q^d$ and $M$, our constructions provide an algorithm for the computation of invariants and holomorphic differentials for the function field $K_{q,M}$. As we show in Section 6, the invariant ring for the algebra of Carlitz generators is not polynomial. This shows that a closed formula for a basis for holomorphic differentials will be quite complicated in general, despite that it can be done when $M$ splits over $\mathbb{F}_q$ \cite{Ward}.

\section{Preliminary results}

\subsection{Inclusions of cyclotomic function fields} We now give the proof of Theorem \ref{inclusion}.

\begin{proof}[Proof of Theorem \ref{inclusion}] For a global field $K$, let $S(L/K)$ denote the collection of places of $K$ which split completely in $L$. By Bauer's theorem \cite[Theorem 11.5.1]{salvador2007topics}, we have for a global function field $K$ and two Galois extensions $L_1$ and $L_2$ of $K$ that $$L_2 \subset L_1 \Leftrightarrow S(L_1/K) \subset S(L_2/K).$$ Thus, the extension $K_{q,M}$ is contained in $K_{q^d,M}$ if, and only if, the places of $\mathbb{F}_q(T)$ which split in $K_{q^d,M}$ also split in $K_{q,M}$. For a place $\mathfrak{p}$ of $\mathbb{F}_q(T)$, we let $\mathfrak{P}$ denote a place of $K_{q,M}$ above $\mathfrak{p}$, $\mathfrak{p}_{q^d}$ a place of $\mathbb{F}_{q^d}(T)$ above $\mathfrak{p}$, and $\mathfrak{P}_{q^d}$ a place of $K_{q^d,M}$ above $\mathfrak{p}$. If $\mathfrak{p}$ denotes a place of $K = \mathbb{F}_q(T)$ which splits completely in $K_{q^d,M}$, then we have $$[\mathfrak{O}_{\mathfrak{P}_{q^d}}/\mathfrak{P}_{q^d}:\mathfrak{o}_{\mathfrak{p}}/\mathfrak{p}] = 1,$$ where $\mathfrak{O}_{\mathfrak{P}_{q^d}}$ denotes the valuation ring for $\mathfrak{P}_{q^d}$ in $K_{q^d,M}$ and $\mathfrak{P}_{q^d}$ its maximal ideal, $\mathfrak{o}_{\mathfrak{p}}$ the valuation ring for $\mathfrak{p}$ in $K$ and $\mathfrak{p}$ its maximal ideal. In the analogous notation, we thus have $$[\mathfrak{O}_{\mathfrak{P}_{q^d}}/\mathfrak{P}_{q^d}:\mathfrak{o}_{\mathfrak{p}_{q^d}}/\mathfrak{p}_{q^d}][\mathfrak{o}_{\mathfrak{p}_{q^d}}/\mathfrak{p}_{q^d}: \mathfrak{o}_{\mathfrak{p}}/\mathfrak{p}] = 1.$$
In particular, it follows that $$[\mathfrak{O}_{\mathfrak{P}_{q^d}}/\mathfrak{P}_{q^d}:\mathfrak{o}_{\mathfrak{p}_{q^d}}/\mathfrak{p}_{q^d}] = 1,$$ and that the place $\mathfrak{p}_{q^d}$ of $\mathbb{F}_{q^d}(T)$ is completely split in $K_{q^d,M}$. Notice that $\mathfrak{p}$ cannot be infinity, as infinity is not split in cyclotomic extensions (the ramification index is the size of the constant field minus one \cite[Theorem 12.4.6]{salvador2007topics}). It follows that the place $\mathfrak{p}_{q^d}$ is associated with an irreducible polynomial $P_{q^d} \in \mathbb{F}_{q^d}[T]$. By cyclotomic function field theory, for example, Proposition 12.5.2 of Villa-Salvador 
\cite{salvador2007topics}, the inertia degree of $\mathfrak{p}_{q^d}$ is equal to the order of $P_{q^d}$ modulo $M$. In fact, we may write $$P = \prod_{i=1}^k P_i$$ in $\mathbb{F}_{q^d}[T]$, where $P$ is associated with $\mathfrak{p}$ in $\mathbb{F}_q[T]$ and the polynomials $P_i$ are the (distinct) factors of $P$ in $\mathbb{F}_{q^d}(T)$. 
We also know that the finite places of $\mathbb{F}_{q^d}$ which divide $M$ ramify in $K_{q^d,M}$, and hence are not completely split. Also, we know that $P_{q^d} = P_i$ for some $i$, but also that $P_{q^d}$ was an arbitrary choice of place, which could be done as the extensions we consider are all Galois. Thus $(P_i,M)=1$ for all $i=1,\ldots,k$, from which it follows that $(P,M) = 1$. By cyclotomic function field theory, we also know that the order of $P$ modulo $M$ in $\mathbb{F}_q[T]$ is equal to the inertia degree of $\mathfrak{p}$ in $K_{q,M}$. For each $i = 1,\ldots,k$, let $o_{h,i}$ denote the order of $P_i$ modulo $M$ in $\mathbb{F}_{q^d}[T]$. Thus $$P_i^{o_{h,i}} \equiv 1 \mod M, \quad i=1,\ldots,k.$$ We may then write $P_i^{o_{h,i}} - 1 = FM$ for some $F \in \mathbb{F}_{q^d}[T]$. 
As the Galois action of $\mathbb{F}_{q^d}(T)/\mathbb{F}_q(T)$ is transitive, it follows that for any $j = 1,\ldots,k$, there exists $\sigma \in \text{Gal}(\mathbb{F}_{q^d}(T)/\mathbb{F}_q(T))$ such that $\sigma(P_i) = P_j$. It follows that $$P_j^{o_{h,i}} - 1 =\sigma(P_i)^{o_{h,i}} - 1 = \sigma(P_i^{o_{h,i}}) - 1 = \sigma(P_i^{o_{h,i}} - 1) = \sigma(FM) = \sigma(F)M,$$ where $\sigma(F) \in \mathbb{F}_{q^d}[T]$ by definition of the Galois action of $\mathbb{F}_{q^d}(T)/\mathbb{F}_q(T)$. It follows by symmetry that the order of each $P_i$ modulo $M$ is the same. We therefore set $o_h := o_{h,i}$ for $i=1,\ldots,k$. It follows from this that $$P^{o_h} = \left(\prod_{i=1}^k P_i\right)^{o_h} =\prod_{i=1}^k P_i^{o_h} \equiv \prod_{i=1}^k 1 = 1 \mod M,$$ and hence the order $o_P$ of $P$ modulo $M$ is at most $o_h$. As the polynomials $P_i$, $i=1,\ldots,k$ are completely split in $K_{q^d,M}$, it follows from the cyclotomic theory that $o_h = 1$, and thus $o_P \leq o_h = 1$. Thus, the place $\mathfrak{p}$ associated with $P$ is completely split in $K_{q,M}$. We have thus shown that if a place $\mathfrak{p}$ of $K = \mathbb{F}_q(T)$ splits completely in $K_{q^d,M}$, then it must also split completely in $K_{q,M}$. Hence, by Bauer's theorem, it follows that $K_{q,M} \subset K_{q^d,M}$.
\end{proof}
Henceforth, as we frequently distinguish between $\mathbb{F}_q[T]$ and $\mathbb{F}_{q^d}[T]$, we denote $A_1 := \mathbb{F}_q[T]$ and $A_d:=\mathbb{F}_{q^d}[T]$. The following properties of Galois extensions may be easily deduced from the previous theorem (see also \cite[Chapter 12]{salvador2007topics}). 
\begin{itemize}
\item 
$\mathrm{Gal}(\mathbb{F}_{q^d}(T)/\mathbb{F}_q(T)) \cong\mathrm{Gal}(\mathbb{F}_{q^d}/\mathbb{F}_q)$
\item 
$\mathrm{Gal}(K_{q,M}/\mathbb{F}_q(T))\cong (A_1/MA_1)^*$.
\item 
$\mathrm{Gal}(K_{q^d,M}/\mathbb{F}_{q^d}(T))\cong (A_d/MA_d)^*$.
\item 
$\mathrm{Gal}(K_{q^d,M}/\mathbb{F}_q(T)) \cong (A_d/MA_d)^*\times \mathrm{Gal}(\mathbb{F}_{q^d}/\mathbb{F}_q)$.
\end{itemize}
We now fix a basis $1=v_1,\ldots,v_d$ of $\mathbb{F}_{q^d}$ as a vector space over $\mathbb{F}_q$, and we write
\[
A_d\cong A_1 \oplus v_2A_1 \oplus \cdots \oplus  v_dA_1.
\]
As $M\in A_1$ by definition, we then have
\[
A_d/M A_d  \cong A_1/MA_1 \oplus v_2A_1/MA_1 \oplus \cdots \oplus v_dA_1/MA_1.
\]
Also, the group $\mathrm{Gal}(K_{q^d,M}/K_{q,M})$ fits in  the short exact sequence
\[
1 \rightarrow \mathrm{Gal}(K_{q^d,M}/K_{q,M}) \rightarrow (A_d/MA_d)^*  \times \mathrm{Gal}(\mathbb{F}_{q^d}/\mathbb{F}_q) \rightarrow (A_1/MA_1)^* \rightarrow 1.
\]
For ease of notation, we define $G_{q^d,M}:=\mathrm{Gal}(K_{q^d,M}/\mathbb{F}_{q^d}(T))$. It is well-known that 
\begin{equation}\label{isomorphism}
	G_{q^d,M} \cong
	\left(A_d/M A_d\right)^*,
\end{equation}
where the Galois action is induced by the Carlitz action by elements of $(\mathbb{F}_{q^d}[T]/M)^*$. The Galois inclusion $\mathbb{F}_q(T) \subset \mathbb{F}_{q^d}(T) \subset K_{q^d,M}$ gives rise to another short exact sequence: 
\[
\xymatrix{
	1 
	 \rightarrow \mathrm{Gal}(K_{q^d,M}
	 /\mathbb{F}_{q^d}(T)) \ar@{=}[d] \ar[r] &
	\mathrm{Gal}(K_{q^d,M}/\mathbb{F}_{q}(T)) \ar[r]^-\pi&
	\mathrm{Gal}(\mathbb{F}_{q^d}(T)/\mathbb{F}_q(T)) \ar@{=}[d] \rightarrow 1
	\\
	 G_{q^d,M} &  &  \mathrm{Gal}(\mathbb{F}_{q^d}/\mathbb{F}_q)
}
\]
By standard arguments of group theory, the Galois group $\mathrm{Gal}(\mathbb{F}_{q^d}/\mathbb{F}_q)$ acts by conjugation on $G_{q^d,M}$ in terms of an inverse section of the map $\pi$. We note that as the group $G_{q^d,M}$ is abelian, this action is well defined, i.e., independent of the choice of the section of $\pi$. 

\subsection{The group $G_{q^d,P}$ as a $\mathrm{Gal}(\mathbb{F}_{q^d}/\mathbb{F}_q)$-module} 
\label{frobAC}
We identify each divisor $D\in (A_d/MA_d)^*$ with the element $\rho_D \in G_{q^d,M}$ under the isomorphism \eqref{isomorphism}. The natural action of
 $\sigma \in \mathrm{Gal}(\mathbb{F}_{q^d}/\mathbb{F}_q)$ on $\rho_D$ is given by conjugation. 
	As such, we have
	\[
	\sigma \rho_D \sigma^{-1}=\rho_{\sigma (D)},
	\]
	where $\sigma(D)$ simply denotes the image of $D\in A_d$ under the natural action of $\sigma$ in $\mathbb{F}_{q^d}$. We have
$$ \rho_D (u):=D *_d u=u^{q^s} + b_{s-1}u^{q^{s-1}}+ \cdots + Du,$$
so that
$$\sigma \rho_D \sigma^{-1}(u) =u^{q^s} + \sigma(b_{s-1})u^{q^{s-1}}+ \cdots + \sigma(D)u,
 $$
where $*_d$ denotes the Carlitz action over $\mathbb{F}_{q^d}(T)$. In what follows, we emphasise that $G_{q,M}=\mathrm{Gal}(K_{q,M}/\mathbb{F}_{q}(T))$ is isomorphic to $\mathrm{Gal}(\mathbb{F}_{q^d}K_{q,M}/\mathbb{F}_{q^d}(T))$, as one may deduce from the tower of fields given in  \eqref{fig:tower}.

\begin{lemma} \label{Kummerfun} The group $H_{q^d,M}=\mathrm{Gal}(K_{q^{d},M}/\mathbb{F}_{q^d}K_{q,M})$ in the short exact sequence 
\begin{equation} \label{sesH}
1 \rightarrow H_{q^d,M} \rightarrow {G}_{q^d,M} \rightarrow G_{q,M} \rightarrow 1
\end{equation}
satisfies for $\sigma\in \mathrm{Gal}(\mathbb{F}_{q^d}/\mathbb{F}_q)$
\begin{enumerate}
	\item $\sigma(H_{q^d,M})=H_{q^d,M}$
	\item $H_{q^d,M}=(\sigma-1)G_{q^d,M}$.
\end{enumerate} 
\end{lemma}
\begin{proof} 
\begin{enumerate}
	\item 
The natural action on subgroups of a Galois groups is the conjugation action.
The field $$\mathbb{F}_{q^d}K_{q,M} =K_{q^d,M}^{H_{q^d,M}}$$ is a subfield of $K_{q^d,M}$ which is invariant, albeit not pointwise, under the action of $\mathrm{Gal}(\mathbb{F}_{q^d}/\mathbb{F}_q)$. Let $\sigma$ be a generator of the cyclic group $\mathrm{Gal}(\mathbb{F}_{q^d}/\mathbb{F}_q)$. As both
 $$\sigma(\mathbb{F}_{q^d} K_{q,M})=K_{q^d,M}^{\sigma H_{q^d,M}\sigma^{-1}}\text{ and }\sigma(\mathbb{F}_{q^d} K_{q,M})=\mathbb{F}_{q^d} K_{q,M},$$ we have that $\sigma H_{q^d,M} \sigma^{-1}=H_{q^d,M}$.
	\item 
	We will prove first that $(\sigma-1)D \in H_{q^d,M}$ for each element $D \in G_{q^d,M}$.
The group $G_{q^d,M}$ consists of classes of invertible elements $D\in \mathbb{F}_{q^d}[T] $ modulo $M$. Let $\sigma\in \mathrm{Gal}(\mathbb{F}_{q^d}/\mathbb{F}_q)$ be a generator of the cyclic Galois group. As $G_{q,M}$ is the Galois group of a geometric extension, it follows that the group $G_{q,M}=G_{q^d,M}/H_{q^d,M}$ is pointwise $\sigma$-invariant. Thus $\sigma(D)/D \in H_{q^d,M}$. 

Since the group $G_{q^d,M}$ is abelian, the map
\[
\xymatrix@R-2pc{
	G_{q^d,M} \ar[r]^\Phi & G_{q^d,M} \\
	\alpha \ar@{|->}[r] & \sigma(\alpha)\alpha^{-1}
}
\]
is a group homomorphism. The kernel of $\Phi$ consists of elements $D\in G_{q^d,M}\cong \left(\mathbb{F}_{q^d}[T]/M \right)^*$ which are left invariant under the action of $\sigma$, hence it is isomorphic to $G_{q,M}$. On the other hand, we have proven that $\mathrm{Im}(\Phi) \subset H_{q^d,M}$. 
Thus $|\mathrm{Im}(\Phi)|=|G_{q^d,M}|/|G_{q,M}|$. Since by definition $|G_{q,M}|=|G_{q^d,M}|/|H_{q^d,M}|$, we obtain $|\mathrm{Im}(\Phi)|=|H_{q^d,M}|$, and we arrive at 
\[
H_{q^d,M}= (\sigma-1) G_{q^d,M},
\]
concluding the proof.
\end{enumerate}
\end{proof}
\begin{remark}
The group $G_{q^d,M}$ is naturally a $\mathrm{Gal}(\mathbb{F}_{q^d}/\mathbb{F}_q)$-module. The group $G_{q,M}$ is the space of coinvariants
\[
G_{q,M}=(G_{q^d,M})_{\mathrm{Gal}(\mathbb{F}_{q^d}/\mathbb{F}_q)}=
\frac{G_{q^d,M}}{(\sigma-1)G_{q^d,M} }.
\]
\end{remark}

 \subsection{Reduction to irreducible factors of $M$}
We now reduce the computation of the structure of the group $H_{q^d,M}$ to the study of the corresponding groups for the irreducible components of $M$. Let $M\in \mathbb{F}_q[T]$ be of degree $k$, with factorization in 
$\mathbb{F}_q[T]$
\[
M=\prod_{j=1}^r M_j^{\alpha_j},
\]
where the polynomials $M_j$ are irreducible, monic and of degree $s_j|d$. In the finite field $\mathbb{F}_{q^d}$, the polynomials $M_j \in \mathbb{F}_q[T]$ factor into linear factors in $\mathbb{F}_{q^d}[T]$. We write
\[
M=\prod_{j=1}^r \prod_{i=1}^{s_j} (T-\rho_{i,j})^{\alpha_j}, \qquad \qquad \rho_{i,j} \in \mathbb{F}_{q^d}.
\]
The field $K_{q,M}$ is the compositum of the fields $K_{q,M_j^{\alpha_j}}$, each of which is, in turn, a subfield of $K_{q^d,M_j^{\alpha_j}}$. This yields the following diagram:
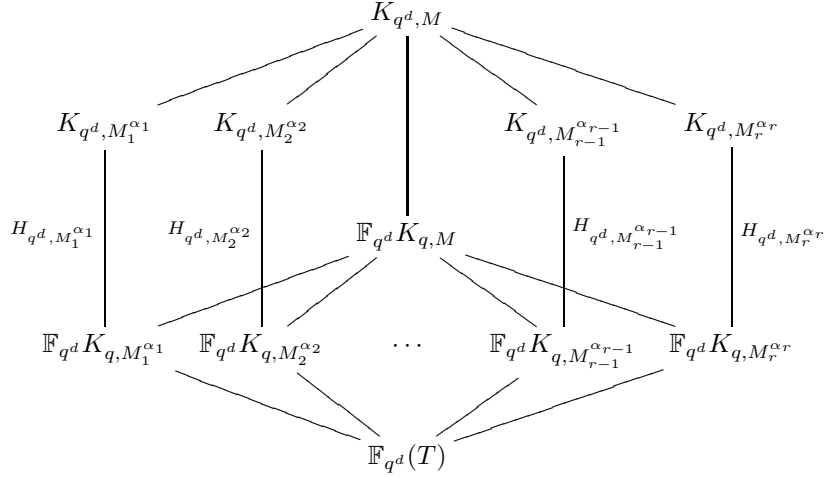
\[
\xymatrix@C=.5pc{
 &  &  K_{q^d,M} \ar@{-}[dd]
\ar@{-}[dll] \ar@{-}[dl] \ar@{-}[dr] \ar@{-}[drr]
  &  &   \\
K_{q^d,M_1^{\alpha_1}} \ar@{-}[dd]_{H_{q^d,M_1^{\alpha_1}}} & 
K_{q^d,M_2^{\alpha_2}} \ar@{-}[dd]_{H_{q^d,M_2^{\alpha_2}}}&  
 &  K_{q^d,M_{r-1}^{\alpha_{r-1}}} 
\ar@{-}[dd]^{H_{q^d,M_{r-1}^{\alpha_{r-1}}}}
  & K_{q^d,M_{r}^{\alpha_r}}
 \ar@{-}[dd]^{H_{q^d,M_r^{\alpha_r}}} 
 \\
  &  & { \mathbb{F}_{q^d} K_{q,M}\ar@{-}[rd] \ar@{-}[rrd] \ar@{-}[ld] \ar@{-}[lld]} & &
 \\
\mathbb{F}_{q^d}K_{q,M_1^{\alpha_1}} & 
\mathbb{F}_{q^d}K_{q,M_2^{\alpha_2}} & 
\cdots & 
\mathbb{F}_{q^d}K_{q,M_{r-1}^{\alpha_{r-1}}}  &
\mathbb{F}_{q^d}K_{q,M_r^{\alpha_r}} 
\\
 & & { \mathbb{F}_{q^d}(T)} 
\ar@{-}[ull] \ar@{-}[ul] \ar@{-}[ur] \ar@{-}[urr]
 &  &
}
\]
We have 
\[
	\mathrm{Gal}(K_{q^d,M}/\mathbb{F}_{q^d}(T))\cong 
	{\bigtimes}_{j=1}^r  \mathrm{Gal}(K_{q^d,M_j^{a_j}}/\mathbb{F}_{q^d}(T))
\]
and 
\[
	H_{q^d,M}= {\bigtimes}_{j=1}^r H_{q^d,M_j^{a_j}}.
\]
The problem of determining the structure of the group $H_{q^d,M}$ may therefore be reduced to determination of each of the groups $H_{q^d,M_i^{\alpha_i}}$. This means that we may assume $M$ to be a power $P^\alpha$ of an irreducible polynomial $P \in \mathbb{F}_q[T]$. Thus, as a consequence of Lemma \ref{Kummerfun}, we now study the group $G_{q^d,P}$, where $P$ is an irreducible polynomial in $\mathbb{F}_q[T]$ of degree $s | d$. The splitting field of $P$ is equal to $\mathbb{F}_{q^s}$. By definition, we have the following short exact sequence:
\begin{equation}
\label{ses-Galois}
1 \rightarrow P_{q^d,P^\alpha} \rightarrow G_{q^d,P^\alpha} \rightarrow G_{q^d,P} \rightarrow 1,
\end{equation}
where 
\[
	P_{q^d,P^\alpha}=\{D \in \mathbb{F}_{q^d}[T] \mod P^\alpha, D\equiv 1 \mod P\}.
\]
The field $K_{q^d,P^\alpha}$ is the compositum of a generalised Artin-Schreier extension with the  Kummer extension $K_{q^d,P}/\mathbb{F}_{q^d}(T)$. The subfield $\mathbb{F}_{q^d} K_{q,P^\alpha}$ has a similar decomposition, into a generalised Artin-Schreier extension with Galois group $\mathbb{Z}/q^{\alpha-1}\mathbb{Z}$  with the Kummer extension $\mathbb{F}_{q^d}K_{q,P}/\mathbb{F}_{q^d}(T)$.
We have the following tower of fields:
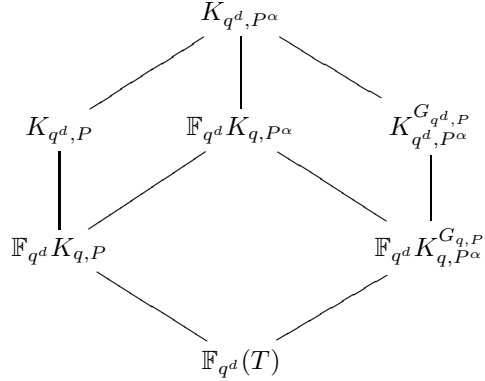
\[
\xymatrix{
	& K_{q^d,P^\alpha} \ar@{-}[dr] \ar@{-}[dl] \ar@{-}[d] & 
	\\
K_{q^d,P} \ar@{-}[d] &  
\mathbb{F}_{q^d}K_{q,P^\alpha}  \ar@{-}[dr] \ar@{-}[ld]
& K_{q^d,P^\alpha}^{G_{q^d,P}}	\ar@{-}[d]
\\
\mathbb{F}_{q^d}K_{q,P}  \ar@{-}[dr]&   
& \mathbb{F}_{q^d}K_{q,P^\alpha}^{G_{q,P}} 
\ar@{-}[dl] \\
& \mathbb{F}_{q^d}(T)
}
\]
\begin{lemma}
\label{Fi}
Let $\sigma\in \mathrm{Gal}(\mathbb{F}_{q^d}/\mathbb{F}_q)$, and for each $i=1,\ldots,\alpha$, let
$$
F_i:=K_{q,P^i}^{G_{q,P}}.
$$ 
Then for each $i=1,\ldots,\alpha$, the generator $\sigma$ of the Galois group $\mathrm{Gal}(\mathbb{F}_{q^d}/\mathbb{F}_q)$ satisfies
\begin{enumerate}
	\item 
 $\sigma(K_{q^d,P^i})=K_{q^d,P^i}$. 
 	\item 
 $\sigma(F_i)=F_i$.
\end{enumerate}
\end{lemma}

\begin{proof} By construction, we have $\sigma(\mathbb{F}_{q^d}(T))=\mathbb{F}_{q^d}(T)$ and
 $\sigma(K_{q^d,P^i})=K_{q^d,P^i}$, for each $i=1,\ldots,\alpha$. The Galois group $\mathrm{Gal}(K_{q^d,P^\alpha}/\mathbb{F}_{q^d}(T))$ is congruent to the direct product of the cyclic group $\mathbb{F}_{q^d}^*$ with the $p$-group $\mathrm{Gal}(F_\alpha/\mathbb{F}_{q^d}(T))$. By definition, we have that $\mathbb{F}_{q^d}(T) \subset \sigma(F_\alpha) \subset K_{q^d,P^\alpha}$, whence $\sigma(F_\alpha)$ corresponds to a Galois group $H$ which is isomorphic to a cyclic group of order $q^d-1$. As there is a unique such subgroup in $\mathrm{Gal}(K_{q^d,P^\alpha}/\mathbb{F}_{q^d}(T))$, it follows that $\sigma(F_\alpha)=F_\alpha$. 

 The result for $F_i$ for each $i=1,\ldots, \alpha-1$ follows by induction.
\end{proof}

We note that property (2) in Lemma \ref{Fi} implies that $(\sigma-1)G_{q^d,P}$ corresponds to a subgroup of $H_{q^d,P}$, and by order comparisons, we obtain that the unique such submodule of $G_{q^d,P}$ is given by the image of the map $\sigma-1$. For a realisation of a Kummer model of $K_{q,P}$, we refer to Section \ref{KummerSection}.

\section{Tame structure}
\label{KummerSection}

In order to understand $H_{q^d,P}$, we will describe the character group of $G_{q^d,P}$ using the torsion of the Carlitz module. For an introduction to Kummer theory of extensions the reader is reffered to \cite{Jacobson2}.
Let $P=\prod_{i=1}^s (T-\rho_i)$ be the decomposition of the irreducible polynomial $P\in \mathbb{F}_q[T]$ in $\mathbb{F}_{q^d}[T]$.
We know by prime decomposition \cite[Chapter 12]{salvador2007topics} that 
\begin{equation}\label{Chinese}
C_{q^d}[P]=\bigoplus_{i=1}^{ s} 
C_{q^d}[T-\rho_i].
\end{equation}
The torsion modules $C_{q^d}[T-\rho_i]$ are defined as 
\[
C_{q^d}[T-\rho_i]:=\{z \in \overline{\mathbb{F}_q(T)}: z^{q^d}+(T-\rho_i)z=0 \}.
\]
Let $\lambda_i$ be a generator of $C_{q^d}[T-\rho_i]$ as an $A_d$-module. Then $\lambda_i$ satisfies the equation
\begin{equation}
	\label{lambdaKummer}
	\lambda_i^{q^d-1}=-(T-\rho_i).
\end{equation}
\begin{lemma}
	\label{Galdact}
	Let $\sigma$ be a generator of the cyclic group $\mathrm{Gal}(\mathbb{F}_{q^d}/\mathbb{F}_q)$. Then
	\[
	\sigma(\lambda_i)=
	\begin{cases}
	\zeta_{\sigma,i} \lambda_{i+1} & \mbox{ if } 1 \leq i< {s}\\
	\zeta_{\sigma,d} \lambda_1 & \mbox{ if } i={s},
	\end{cases}
	\]
	where $\zeta_{\sigma,i}$ is a $(q^{d}-1)$st root of unity, which depends on both $\sigma$ and $i$.
\end{lemma}
\begin{proof}
Consider the action of  $\sigma$ on \eqref{lambdaKummer}. If $i< s$, then
\[
(\sigma \lambda_i)^{q^d-1}=-\sigma(T-\rho_i)=-(T-\sigma(\rho_i))=-(T-\sigma_{i+1})=\lambda_{i+1}^{q^d-1}.
\]	
The result follows in this case. The proof for $i=s$ is the same, except that $\sigma(\rho_s)=\rho_1$. 
\end{proof}
On the other hand, the action of $G_{q^d,P}$ on $\lambda_i$ is given by multiplication by elements in $\mathbb{F}_{q^d}^*$. Indeed, 
if $f\in \mathbb{F}_{q^d}[T]$ and $(f,P)=1$, then $f(\rho_i) \neq 0$ for all $i=1,\ldots,s$. By definition of the Carlitz action $*_d$, 
and the fact that $\lambda_i$ is a $(T-\rho_i)$-torsion point it follows that 
\begin{equation} \label{*daction}
\sigma_f \lambda_i= f *_d \lambda_i=f(\rho_i)\cdot \lambda_i. 
\end{equation}
We note that $f(\rho_i) \in \mathbb{F}_{q^d}$, whence $f(\rho_i)^{q^d-1}=1$. 
\begin{definition}
	Let $\zeta$ be a fixed choice of primitive $(q^{d}-1)$st root of unity. We define the dual elements $\{\sigma_k^*\} \in G_{q^d,P}^*$ such that 
	\[
	\sigma_k^*(\sigma_\ell)=
	\begin{cases}
		1 & \mbox{ if } k\neq \ell \\
		\zeta & \mbox{ if } k=\ell.
	\end{cases}
	\]
\end{definition}

\begin{remark}
As $s \mid d$ we have that $(q^s-1) \mid (q^d-1)$. The element $$\zeta_1:=\zeta^{\frac{q^d-1}{q^s-1}}$$ generates a cyclic subgroup of $G_{q^d,P}$ of order $q^s-1$.
\end{remark}

\begin{lemma} For each $i_0 = 1,\ldots,s$, consider the polynomials 
	\[
	f_{i_0}(x):=\left(1-
	  \prod_{\substack{i=1 \\ i\neq i_0}}^{s} (x-\rho_i)\right)=1-\frac{P(x)}{x-\rho_{i_0}} \in \mathbb{F}_{q^d}[x].
	\] 
Then
	\[
	\sigma_{f_{i_0}}*_d \lambda_j= 
	\left(1- 
	\prod_{\substack{i=1 \\ i\neq i_0}}^{{s}}  (\rho_{j}-\rho_i) 
	\right)
	\cdot \lambda_j.
	\]
\end{lemma}
\begin{proof} The proof follows immediately from \eqref{*daction}. \end{proof}
\begin{lemma} For each $i_0=1,\ldots,s$, consider the elements
\[
Z_{i_0}:=\left(1- 
	\prod_{\substack{i=1 \\ i\neq i_0}}^{s} (\rho_{i_0}-\rho_i) 
	\right).
\]
Then
\[
Z_{i_0}=Z_1^{q^{i_0-1}}.
\]
Furthermore, the elements $Z_{i_0}$ are primitive $(q^{s}-1)$st roots of unity.
\end{lemma}
\begin{proof}
The number of generators of a cyclic group of order $q^{s}-1$ is equal to $\phi(q^{s}-1)$, where $\phi$ denotes the Euler totient function. Furthermore, it is well-known that ${s} \mid \phi(q^{s}-1)$. We find, 
 recall that  $\rho_1=\zeta_1$ and $\rho_i=\zeta_1^{q^{i-1}}$ for $1\leq i \leq {s}-1$,
	\begin{align*}
	Z_{i_0} & =  \left(1-\prod_{\substack{i=1 \\ i\neq i_0}}^{{{s}}}
	   (\zeta_1^{q^{i_0-1}}-\zeta_1^{q^{i-1}})
	\right) \\
	&=  \left(1-\prod_{i=2 }^{{{s}}}
	   (\zeta_1^{q^{i_0-1}}-\zeta_1^{q^{i+i_0-2}})
	   \right)\\
	& =  \left(1-\prod_{i=2 }^{{{s}}}
	   (\zeta_1-\zeta_1^{q^{i-1}})
	   \right)^{q^{i_0-1}} 
	 =  Z_1^{q^{i_0-1}},
	\end{align*}
	which concludes the proof.
\end{proof}
Recall that  each generator $\lambda_i$ of
$C_{q^d}[T-\rho_i]$ satisfies 
\eqref{lambdaKummer}. 
The elements $\lambda_i$ define characters $\sigma_{\lambda_i}^*=\chi_{\lambda_i}$ by Kummer theory, which are given by
\[
\sigma_{\lambda_i}^*(\sigma)=\lambda_i^{\sigma}\lambda_i^{-1},
\]
for each $\sigma \in G_{q^d,P}$. In particular, we find that
\[
\sigma_{\lambda_i}^*(\sigma_{f_j})=\lambda_{i}^{\sigma_{f_j}} \lambda_{i}^{-1}=
\begin{cases}
Z_1^{q^{j-1}} & \mbox{ if } i= j \\
1  & \mbox{ if } i \neq j.
\end{cases}
\]
We now write $Z_1=\zeta_1^\alpha$ for some $\alpha \in \mathbb{N}$,
$(\alpha,q^{{s}}-1)=1$. The character group of $G_{q^d,P}$ is non-canonically isomorphic to $G_{q^d,P}$. Moreover, letting 
{$n = q^s - 1$}, the quotient map in \eqref{sesH} gives  
\[G_{q^d,P} \rightarrow \mathbb{Z}/n \mathbb{Z} \rightarrow 1,\]
which by duality yields the injection of the cyclic group $\mathbb{Z}/n\mathbb{Z}$
\[
1 \rightarrow \mathbb{Z}/n\mathbb{Z} \rightarrow G_{q^d,P}^*.
\]
For each $i=1,\ldots,{{s}}$, let $\sigma_i$ be chosen generators of the $i$th direct summand in the decomposition given in \eqref{Chinese} of the group $G_{q^d,P}$, and consider a dual basis $\sigma^*_i$ of $G_{q^d,P}^*$ such that 
\[\sigma_i^* \sigma_j=\delta_{ij} \rho_1{=\delta_{ij}\zeta_1 },\] 
where $\delta_{ij}$ is Kronecker's $\delta$. An injection $$\iota: \langle g \rangle=\mathbb{Z}/n\mathbb{Z}\rightarrow G_{q^d,P}^*$$ is described by giving the coordinates of the generator $g$, i.e.,
\begin{equation} \label{define-chara}
\iota(g)=\prod_{i=1}^{{s}} \left( \sigma_i^*\right)^{ 
\frac{q^d-1}{q^s-1}b_i}.
\end{equation}
Furthermore, the element $\iota(g)$ has order $n$ if, and only if, at least one of the integers $b_i$ is prime to $n$. Also, the character $\iota(g)$ given in \eqref{define-chara} is associated (via the Kummer correspondence) to the element 
\begin{equation} \label{L}
  L=\prod_{i=1}^{{s}}
  \lambda_i^{\frac{q^d-1}{q^s-1}b_i}, \mbox{ which satisfies } L^{q^s-1}=
  (-1)^s
 \prod_{i=1}^{{s}} (T-\rho_i)^{b_i}.  
\end{equation}
\begin{lemma} \label{KummerGen} For each $i=1,\ldots,{{s}}$, the exponent $b_i$ in \eqref{L} may be explicitly and recursively determined.
\end{lemma}
\begin{proof}
For the Frobenius generator $\sigma$ of the cyclic group $\mathrm{Gal}(\mathbb{F}_{q^{{s}}}/\mathbb{F}_q)$ and $L$ as defined in \eqref{L}, the element $\sigma(L)$ generates the Kummer extension  and by the theory of Kummer extensions it  has the form
\[
\sigma(L)=L^{\mu} a^{q^{{s}}-1}, 
\]
for some $\mu\in \mathbb{N}$ such that $(\mu,q^{{s}}-1)=1$ 
and some $a\in \mathbb{F}_{q^d}(T)$. This implies that 
\[
\lambda_1^{b_s}\prod_{i=2}^{{{s}}} \lambda_i^{b_{i-1}}
=\prod_{i=1}^{{s}} \lambda_i^{b_i \mu} a^{q^{{s}}-1},
\]
which in turn implies for each $i = 2,\ldots,{s}$ that
\[
b_{i-1}=\mu b_i \mod q^{{s}}-1 
\]
and also that
\[
b_{{s}}=\mu b_1 \mod q^{{s}}-1.
\]
We let $\mu'$ be the inverse of $\mu$ modulo $q^d-1$. This yields
\begin{align}
	\notag b_1 \mu' & \equiv  b_2 \mod q^{{s}}-1 \\
	\notag b_2 \mu' & \equiv   b_3 \mod q^{{s}}-1 \\
  \label{lasteq1}	\cdots &  \cdots  \\
	\notag b_{{{s}}-1} \mu' &\equiv  b_{{s}} \mod q^{{s}}-1 \\
	\notag b_{{s}} \mu' & \equiv b_1 \mod q^{{s}}-1. 
\end{align} 
Therefore, $b_i = (\mu')^{i-1} b_1$ for each $i=2,\ldots,{{s}}$.
As $b_{{s}}=(\mu')^{{{s}}-1}b_1$ \eqref{lasteq1} implies that 
\[
(\mu')^{{s}} \equiv 1 \mod q^{{s}}-1.
\]
We can thus select $\mu'=q$ to obtain the appropriate value of
 $\mu' \mod (q^{{s}}-1)$. It may also be assumed without loss that $b_1=1$. \end{proof}

\noindent We have thus proven the following result.
\begin{proposition}
The model of the function field $\mathbb{F}_{q^d} K_{q,P}$ is given by the Kummer extension:
\begin{equation} \label{TameKummer}
L^{q^{{s}}-1}=
(-1)^s
\prod_{i=1}^{{s}} 
\left( 
T-\zeta_1^{q^{i-1}} 
\right)^{q^i}.	
\end{equation} 
\end{proposition}
The Galois module structure for differentials of such extensions is known by the work of Boseck \cite{Boseck}. 
In Kummer extensions of $\mathbb{F}_{q^d}(T)$ of the form 
$y^n=\prod_{\nu=1}^d (T-a_i)^{b_i}$, $a_i\in \mathbb{F}_{q^d}$ the places $T-a_i$ are ramified with ramification index
\[
	e_i=\frac{q^{{s}}-1}{(q^{{s}}-1,b_i)}, 
\]
see \cite{Ko}.
Furthermore, by the theory of Carlitz torsion modules, we may also easily see that there is ramification over all places corresponding to the linear factors $T-\rho_i$ of $P$ in $A_d$. It follows that $(b_i,q^{{s}}-1)=1$, for each $i=1,\ldots,{{s}}$. As there is ramification of degree $q-1$ at infinity in $K_{q,P}/\mathbb{F}_q(T)$, we obtain
\[
\sum_{\nu=1}^{{s}} b_i=1+q+\cdots +q^{{{s}}-1}.
\]
We have the following tower of fields and ramified places over $\mathbb{F}_{q^d}(T)$: 
\[
\xymatrix@C=.8pc{
K_{q^d,M} \ar@{-}[d] &  
B_{1,1} \cdots B_{1,q^s-1} \ar@{-}[d]^{\frac{q^d-1}{q^s-1}} & \cdots & 
B_{s,1} \cdots B_{s,q^s-1} \ar@{-}[d]^{\frac{q^d-1}{q^s-1}}
& Q_1,\ldots Q_{t} \ar@{-}[d]^{\frac{q^d-1}{q-1}} 
\\
 \mathbb{F}_{q^d} K_{q,M} \ar@{-}[d]_{\mathbb{F}_{q^{{s}}}^*}  & B_{1}  &  \cdots & B_{{s}} &
 B_{\infty,1}, \ldots B_{\infty,\frac{q^{{s}}-1}{q-1}}  \\
\mathbb{F}_{q^d}(T) & P_1 \ar@{-}[u]^{q^{{s}}-1} & \cdots & P_{{s}} \ar@{-}[u]^{q^{{s}}-1} & P_\infty \ar@{-}[u]^{q-1} 
}
\]
{
\begin{remark}
\label{no-rem}
In the special case $s=d$, there is no ramification over $B_1,\ldots,B_s$ in the 
extension $K_{q^d,P}/\mathbb{F}_{q^d} K_{q,P}$.
\end{remark}
}

\section{Wild structure}

We now proceed to the case $M=P^\alpha$, where $P \in \mathbb{F}_q[T]$ is again an irreducible polynomial of degree $s\mid d$. By previous arguments, it is easily seen that the abelian Galois group $\mathrm{Gal}(K_{q^d,P^\alpha}/\mathbb{F}_q(T))$ may be written as the direct product
\[
\mathrm{Gal}(K_{q^d,P^\alpha}/\mathbb{F}_q(T))\cong  G_{q^d,P}  \times P_{q^d,P^\alpha},
\]
where 
\[
P_{q^d,P^\alpha}=\{D \in \mathbb{F}_{q^d}[T] \mod P^\alpha, D\equiv 1 \mod P\}.
\]
There is a very precise way to describe the group $P_{q^d,P^\alpha}$, which is also used in the elementary proof of the Kronecker-Weber theorem for rational function fields \cite{Salas-Torres}.
\begin{proposition}[Proposition 5.1 \cite{Salas-Torres}]
\label{ab-p-grp}
Let $r$ be a positive integer. The group $P_{q^r,P^\alpha}$ is an abelian $p$-group. Let $v_{q^r,n}(\alpha)$ denote the number of cyclic groups of order $p^n$ in the decomposition of $P_{q^r,P^\alpha}$, where $s=\deg P$. Then\begin{equation*}
v_{q^r,n}(\alpha)=\frac{q^{rs\big(\alpha-\mayorchico{\alpha}
{p^n}\big)}-q^{rs\big(\alpha-\mayorchico{\alpha}
{p^{n-1}}\big)}}{p^{n-1}(p-1\big)}=
\frac{q^{rs\big(\alpha-\mayorchico{\alpha}
{p^{n-1}}\big)}\big(q^{rs\big(\mayorchico{\alpha}{p^{n-1}}
-\mayorchico{\alpha}{p^{n}}\big)}-1\big)}{p^{n-1}(p-1)},
\end{equation*}
where $\lceil x\rceil$ denotes the ceiling function on $\mathbb{Q}$, i.e., the minimum integer greater than or equal to $x$.
\end{proposition}
In particular, Proposition \ref{ab-p-grp} holds for both $P_{q,P^\alpha}$ and $P_{q^d,P^\alpha}$ by setting $r = 1$ and $r = d$, respectively. For the Kummer extensions $P_{q,P}$ and $P_{q^d,P}$, we must realise the Galois group 
\[
H'_{q^d,P^\alpha}:=\mathrm{Gal}\left(K_{q^d,P^\alpha}^{G_{q^d,P}}/\mathbb{F}_{q^d}K_{q,P^\alpha}^{G_{q,P}}\right).
\]
For this, we obtain the exact sequence
\[
0 \rightarrow H'_{q^d,P^\alpha} \rightarrow 
\mathrm{Gal}(K_{q^d,P^\alpha}^{G_{q^d,P}}/\mathbb{F}_{q^d}(T)) 
\rightarrow \mathrm{Gal}(\mathbb{F}_{q^d}K_{q,P^\alpha}^{G_{q,P}} / \mathbb{F}_{q^d}(T)) \rightarrow 0,
\]
where the structure of the abelian $p$-group $$P_{q,P^\alpha}=
\mathrm{Gal}(\mathbb{F}_{q^d} K_{q,P^\alpha}^{G_{q,P}} / \mathbb{F}_{q^d}(T))$$ is given by Proposition \ref{ab-p-grp}. 
\begin{corollary} $H'_{q^d,P^{\alpha}}=(\sigma-1)P_{q^d,P^\alpha}$. \end{corollary}
\begin{proof} This follows in the same manner as Lemma \ref{Kummerfun}, as $\sigma(H'_{q^d,P^\alpha})=H'_{q^d,P^\alpha}$ and $(\sigma-1)D \in H'_{q^d,P^\alpha}$ for every $D\in P_{q^d,P^\alpha}$. \end{proof}

In the extension $K_{q^d,P^\alpha}/\mathbb{F}_{q^d} K_{q,P^\alpha}$, there is generally wild ramification: According to \cite[Proposition 12.4.5]{salvador2007topics}, the place at infinity is ramified in the extension $\mathbb{F}_{q^d} K_{q,P^\alpha}$ with index $q-1$ and in the extension $K_{q^d,P^\alpha}$ with ramification index $q^d-1$. Moreover, the ramification degree of $P$ in the extension $\mathbb{F}_{q^d}K_{q,P^\alpha} /\mathbb{F}_{q^d}(T)$ is equal to {$q^{s\alpha}-q^{s(\alpha-1)}$},
while the ramification degree of $P$ in extension $K_{q^d,P^\alpha}/\mathbb{F}_{q^d}(T)$ is given by $q^{d\alpha}-q^{d(\alpha-1)}$ \cite[Proposition 12.3.14]{salvador2007topics},
whence the ramification in extension $K_{q^d,P^\alpha}/\mathbb{F}_{q^d} K_{q,P^\alpha}$ is given by
\[
{
e=\frac{q^{d\alpha}-q^{d(\alpha-1)}}{q^{s\alpha}-q^{s(\alpha-1)}}=q^{(d-s)(\alpha-1)}\frac{q^{d}-1}{q^s-1}.
}
\]
We thus obtain the following diagram, where $s=(q^{d\alpha}-q^{d(\alpha-1)})/(q-1)$:
\[
\xymatrix@C=1.3pc{
	K_{q^{d},P^\alpha} \ar@{-}[d]^H &  P' \ar@{-}[d]^e
	& \infty_{1,1}, \ldots, \infty_{1,q-1} \ar@{-}[d]^{\frac{q^d-1}{q-1}}
	& & \infty_{s,1}, \ldots, \infty_{s,q-1} \ar@{-}[d]^{\frac{q^d-1}{q-1}}& \\
    \mathbb{F}_{q^d}K_{q,P^\alpha}  \ar@{-}[d] & P_1 \ar@{-}[d] 
    & \infty_1 \ar@{-}[rd]^{q-1} & \ldots  &\infty_{s} \ar@{-}[ld]_{q-1}  \\
    \mathbb{F}_{q^d}(T)  & P  &  & \infty  & 
}
\]
The irreducible polynomial $P$ factors in $\mathbb{F}_{q^d}$ as
\begin{equation} \label{factorisation}
P(T)=\prod_{i=1}^{s}(T-\zeta_1^{q^{i-1}}),
\end{equation}
where $\zeta_1$ is a primitive $(q^s-1)$st root of 1. We denote $P_i = T-\zeta_1^{q^{i-1}}$ ($i=1,\ldots,s$).

\begin{lemma}
The quotient  $A_d/P^\alpha$ is $\mathrm{Gal}(\mathbb{F}_{q^d}/\mathbb{F}_q)$-equivariantly isomorphic to the direct sum of vector spaces:
\[
A_d/P^\alpha=  \left(A_1/P \right)^\alpha.
\]
The group of units $(A_d/P^\alpha)^*$ satisfies
\[
(A_d/P^\alpha)^*= (A_d/P)^* \oplus \left( A_d/P\right)^{\alpha-1}.
\]
\end{lemma}
\begin{proof}
It is clear that the following sequence is exact:
\begin{equation} \label{reduce-power}
0
\rightarrow 
P^{\alpha-1}/P^\alpha
\rightarrow 
A_1/P^\alpha
\rightarrow 
A_1/P^{\alpha-1}
\rightarrow 
1.
\end{equation}
Therefore, we can prove by induction that 
\[
(A_d/P^\alpha)= A_d/P \oplus P/P^2 \oplus P^2/P^3 \oplus \cdots \oplus 
P^{\alpha-1}/P^\alpha.
\]
On the other hand, the map 
\begin{align*}
A_d/P & \rightarrow  P^i/P^{i+1} \\
 f \mod P & \mapsto  f\cdot P^i \mod P^{i+1} 
\end{align*}
is a $\mathrm{Gal}(\mathbb{F}_{q^d}/\mathbb{F}_q)$-equivariant isomorphism, since $P^i$ is a $\mathrm{Gal}(\mathbb{F}_{q^d}/\mathbb{F}_q)$-invariant element. 
\end{proof}
Let $f \in A_d/P^\alpha$. Consider the class $f_0$ of $f$ modulo $P^{\alpha-1}$, i.e., 
\[
f \mod P^\alpha \equiv f_0 + P^{\alpha-1} f_1 \mod P^\alpha,  \qquad f \mod P^{\alpha-1} \equiv f_0 \mod P^{\alpha-1}.
\]
We consider now the multiplication
\begin{align*}
f\cdot g &\equiv  (f_0 + P^{\alpha-1} f_1 ) (g_0+ P^{\alpha-1}g_1) \\
& \equiv  f_0 g_0 + ( \bar{f}_0 g_1+ f_1 \bar{g}_0 )P^{\alpha-1} \mod P^{\alpha},
\end{align*}
where $\bar{f}_0,\bar{g}_0$ denote the classes of $f_0,g_0 \mod P$. It is clear by induction that $f$ is invertible if, and only if, $\bar{f}_0\equiv f \mod P$ is invertible. We define the following filtration:
\begin{equation}
\label{Nt-def}
N_t:=\{ D \mod P^\alpha : D \equiv 1 \mod P^t\}.
\end{equation}
The group $N_1$ is the wild part of $(A_d/P^\alpha)^*$. We have the following short exact sequence 
\[
1 \rightarrow N_{t+1} \rightarrow N_{t} \rightarrow P^t/P^{t+1} \rightarrow 1,
\]
where the group structure on $N_t$ is multiplicative while the structure on $P^t/P^{t+1}$ is additive. The wild part of  $H_{q^d,P^\alpha}$ consists of elements of the form $\sigma(D)D^{-1}$ in $N_1$, where $D\in N_1$. The filtration of $N_1$
\[
N_1 \supset N_2 \supset \cdots N_i \supset \cdots
\]
induces a filtration on $H_{q^d,P^\alpha}$:
\[
H^1=H_{q^d,P^\alpha}\cap N_1 \supset H^2=H_{q^d,P^\alpha} \cap N_2 \supset \cdots  \supset H^i=H_{q^d,P^\alpha} \cap N_i \supset \cdots   
\]
We observe that 
\[
\frac{H^i}{H^{i+1}}=\{ \sigma(D)D^{-1} \mod H^{i+1}: D \in H^i\},
\] 
and the latter group can be identified with the image of the operator $(\sigma-1)(f)$ in the additive group $P^i/P^{i+1}\cong A_d/P$.  

\begin{lemma} \label{H-filt}
The space $A_d/P \cong \oplus_{i=1}^s \mathbb{F}_{q^d}$. An element $f\in A_d/P$ is mapped to the coordinates $(f(\zeta),f(\zeta^q),\ldots,f(\zeta^{q^{s-1}})) \in \oplus_{i=1}^s \mathbb{F}_{q^d}$. The action of the operator $\sigma-1$ on the coordinates $(x_1,\ldots,x_s) \in \oplus_{i=1}^s \mathbb{F}_{q^d}$ is given by:
\[
\xymatrix{
(x_1,\ldots,x_s)  \ar[r]^-{\sigma-1} & (x_s^{q}-x_1, x_1^q-x_2, \ldots, 
x_{s-1}^q-x_s).
}
\]
\end{lemma} The kernel of $\sigma-1$ consists of elements $(x_1,\ldots,x_s)$ such that 
\[
x_1\in \mathbb{F}_{q^d}: x_i=x_1^{q^{i-1}}, \; i=2,\ldots, s.
\]
Clearly this set has $q^d$ elements since all elements are determined by the value of $x_1$. Therefore, the image of $\sigma-1$ has exactly $q^{d(s-1)}$ elements. 

\begin{remark}
Given an $H_{q^d,P^\alpha}$-module $\mathcal{M}$, the space of invariants $\mathcal{M}^H$ is given by 
\[
\mathcal{M}^{H_{q^d,P^\alpha}}= \left( \left( \mathcal{M}^{H_{q^d,P^\alpha}^\alpha} \right)^{H_{q^d,P^\alpha}^{\alpha-1}/H_{q^d,P^\alpha}^\alpha } \cdots \right)^{H_{q^d,P^\alpha}^1/H_{q^d,P^\alpha}^2}.
\]
Since $H^\alpha=H^\alpha/H^{\alpha+1}$, we can apply recursively the computation of Lemma \ref{H-filt} in order to compute $H_{q^d,P^\alpha}$-invariants.
\end{remark}

We now turn to ramification groups and the computation of the different. For each $i \in \mathbb{Z}$ with $i \geq -1$, the
 $i$th ramification group of $\mathfrak{P}|\mathfrak{p}$ is defined as $$
G_{i}(\mathfrak{P}|\mathfrak{p}) = \{\sigma \in G_{q^d,P^\alpha}\;| \; v_\mathfrak{P}(\sigma(x) - x) \geq i+1\text{ for all } x \in \mathfrak{O}_\mathfrak{P}\},$$ where $\mathfrak{O}_\mathfrak{P}$ denotes the valuation ring at $\mathfrak{P}$. We denote $G_i = G_i(\mathfrak{P}|\mathfrak{p})$. 

\begin{proposition} \label{filtration} \begin{enumerate} \item The groups $N_k$ defined in \eqref{Nt-def} have order $q^{ds(\alpha-k)}$ for $1\leq k \leq n$ and correspond to the upper ramification filtration at $\mathfrak{P}|\mathfrak{p}$. \item We have
\begin{itemize}
	\item $G_0= G_{q^d,P^\alpha}$,
	\item $G_i= N_k$ for all $q^{d(k-1)} \leq i \leq q^{dk} - 1$ and $1 \leq k \leq \alpha-1$, and  
	\item $G_i = N_\alpha = \mathrm{Id}$ for all $i \geq q^{d(\alpha-1)}$.
\end{itemize}\end{enumerate} \end{proposition}
\begin{proof} \begin{enumerate} \item \cite[Prop. 2.2]{Keller}. \item This follows by the relation between upper and lower ramification filtrations \cite[IV. sec. 3]{SeL}. \end{enumerate} \end{proof}

By Proposition \ref{filtration}, the lower ramification filtration is given by 
\begin{align*}
G_0 & >   G_1=\cdots=G_{q^{d}-1}=N_1   \\
    & >  G_{q^{d}}=\cdots=G_{q^{2d}-1}=N_2  \\
    & >  G_{q^{2ds}}=\cdots = G_{q^{3d}-1}=N_3 \\
    & >  \qquad \qquad \cdots  \\
    & >  G_{q^{d(\alpha-2)}}
=\cdots=G_{q^{d(\alpha-1)}}=N_{\alpha-1} \\
& >  \{1\}.
\end{align*}
The ramification filtration for the group $H_{q^d,P^\alpha}$ may be found by intersecting $H_{q^d,P^\alpha}$ with $G_i$, so that
\[
H_i:=G_i \cap H_{q^d,P^\alpha},\qquad |H_i|=
\begin{cases}
(q^{d}-1)^{s-1} q^{d(s-1)(\alpha-1)} & \text{ if } i=0 \\
q^{d(s-1)(\alpha-i)} & \text{ if } i \geq 1.
\end{cases}
\]
We now determine the different $\mathfrak{D}_{K_{q^d,P^\alpha}/\mathbb{F}_{q^d} K_{q,P^\alpha}}$ of $K_{q^d,P^\alpha}/\mathbb{F}_{q^d} K_{q,P^\alpha}$.
\begin{proposition} 
\label{prop:47}
$$\mathfrak{D}_{K_{q^d,P^\alpha}/\mathbb{F}_{q^d} K_{q,P^\alpha}} = \prod_{i=1}^s  \prod_{\mathfrak{P}|\wp_i} \mathfrak{P}^A  \prod_{\mathfrak{B}|\mathfrak{p}_\infty} \mathfrak{B}^B,$$ 
where $s=\deg P$, $s\mid d$, 
\begin{align}
\label{Aeq}
 A &=({\alpha q^{d\alpha}- (\alpha+1)q^{d(\alpha-1)}}) - ( {\alpha q^{s\alpha} - (\alpha+1)q^{s(\alpha-1)}})\frac{q^{d\alpha } - q^{d(\alpha -1)}}{q^{\alpha s}-q^{(\alpha -1)s}} \\\notag & = \frac{q^{d (\alpha -1)} \left(q^d-q^s\right)}{q^s-1},  \end{align}

and 
\begin{equation}
\label{Beq}
B = (q^d - 2) - (q-2)  \frac{q^d-1}{q-1} = q\left(\frac{q^{d-1} - 1}{q-1}\right).
\end{equation}
\end{proposition}
 \begin{proof}
As the extension $K_{q^d,P}/\mathbb{F}_q(T)$ is separable, we may employ \cite[Theorem 5.7.15]{salvador2007topics}: Among separable extensions $K \subset L \subset M$ of global fields, we have the functorial identity \begin{equation} \label{differentfunct}\mathfrak{D}_{M/K} = \mathfrak{D}_{M/L} \text{con}_{L/M}(\mathfrak{D}_{L/K}), 
\end{equation} 
where $\mathrm{con}_{L/M}$ denotes the conorm map of the corresponding fields $L,M$, see \cite[5.3]{salvador2007topics}.
We consider the two towers $$\mathbb{F}_q(T) \subset K_{q,P^\alpha } \subset \mathbb{F}_{q^d} K_{q,P^\alpha } \subset K_{q^d,P^\alpha } \quad \text{ and }\quad \mathbb{F}_q(T) \subset \mathbb{F}_{q^d}(T) \subset K_{q^d,P^\alpha }.$$ We may now proceed with computations within each of these towers.
\begin{enumerate}

\item $\mathbb{F}_q(T) \subset K_{q,P^\alpha }$. The different $\mathfrak{D}_{K_{q,P^\alpha }/\mathbb{F}_q(T)}$ is given by \cite[prop. 12.7.1]{salvador2007topics}
$$\mathfrak{D}_{K_{q,P^\alpha }/\mathbb{F}_q(T)} = \mathfrak{P}^{\alpha q^{s\alpha } - (\alpha +1)q^{s(\alpha -1)}} \prod_{\mathfrak{B}|\mathfrak{p}_\infty} \mathfrak{B}^{q-2}, $$ where $\mathfrak{P}$ is the unique place of $K_{q,P^\alpha }$ above the place $\mathfrak{p}$ of $\mathbb{F}_q(T)$ associated with $P$.
\item $ \mathbb{F}_{q^d} K_{q,P^\alpha }/ K_{q,P^\alpha }$ is unramified at all places of $K_{q,P^\alpha }$, whence $$\mathfrak{D}_{\mathbb{F}_{q^d} K_{q,P^\alpha } /K_{q,P^\alpha }}  = (1).$$
\item $\mathbb{F}_{q^d}(T)/\mathbb{F}_q(T)$ is also unramified at all places of $\mathbb{F}_q(T)$, whence $$\mathfrak{D}_{\mathbb{F}_{q^d}(T)/\mathbb{F}_q(T)}  = (1).$$ 
\item For the extension $K_{q^d,P^\alpha }/\mathbb{F}_{q^d}(T)$: As $s | d$ and the polynomial $P$ thus splits completely in $\mathbb{F}_{q^d}[T]$, we denote each linear factor of $P$ in $\mathbb{F}_{q^d}[T]$ by $\wp_i$, so that $P = \prod_{i=1}^s \wp_i$. We obtain \cite[Thm. 12.7.2]{salvador2007topics}
$$\mathfrak{D}_{K_{q^d,P^\alpha }/\mathbb{F}_{q^d}(T)} = \left[\prod_{i=1}^s  \prod_{\mathfrak{P}|\wp_i} \mathfrak{P} \right]^{\alpha q^{d\alpha }- (\alpha +1)q^{d(\alpha -1)}}  \prod_{\mathfrak{B}|\mathfrak{p}_\infty} \mathfrak{B}^{q^d-2}.$$
\end{enumerate}

By \eqref{differentfunct}, we may therefore write $\mathfrak{D}_{K_{q^d,P^\alpha }/\mathbb{F}_{q^d} K_{q,P^\alpha }}$ as
\begin{align*} \notag 
\mathfrak{D}^{-1}&_{K_{q^d,P^\alpha }/\mathbb{F}_{q^d} K_{q,P^\alpha }} 
= \mathfrak{D}_{K_{q^d,P^\alpha }/\mathbb{F}_q(T)}^{-1} \text{con}_{\mathbb{F}_{q^d}K_{q,P}/K_{q^d,P^\alpha }}(\mathfrak{D}_{\mathbb{F}_{q^d} K_{q,P^\alpha }/\mathbb{F}_q(T)}) 
\\
\notag  &= \left(\mathfrak{D}_{K_{q^d,P^\alpha }/\mathbb{F}_{q^d}(T)} \text{con}_{\mathbb{F}_{q^d}(T)/K_{q^d,P^\alpha }}(\mathfrak{D}_{\mathbb{F}_{q^d}(T)/\mathbb{F}_q(T)})\right)^{-1} 
\\
 \notag & \qquad \qquad \qquad \times 
\text{con}_{\mathbb{F}_{q^d} K_{q,P^\alpha }/K_{q^d,P^\alpha }}(\mathfrak{D}_{\mathbb{F}_{q^d} K_{q,P^\alpha }/\mathbb{F}_q(T)}) 
\\\notag &= 
\mathfrak{D}^{-1}_{K_{q^d,P^\alpha }/\mathbb{F}_{q^d}(T)}\text{con}_{\mathbb{F}_{q^d} K_{q,P^\alpha }/K_{q^d,P^\alpha }}(\mathfrak{D}_{\mathbb{F}_{q^d} K_{q,P^\alpha }/\mathbb{F}_q(T)})
\\\notag &= \mathfrak{D}^{-1}_{K_{q^d,P^\alpha }/\mathbb{F}_{q^d}(T)}\text{con}_{\mathbb{F}_{q^d} K_{q,P^\alpha }/K_{q^d,P^\alpha }}( \mathfrak{D}_{\mathbb{F}_{q^d} K_{q,P^\alpha }/K_{q,P^\alpha }} 
\\
 \notag & \qquad \qquad \qquad \qquad \times
\text{con}_{K_{q,P^\alpha }/\mathbb{F}_{q^d} K_{q,P^\alpha }}(\mathfrak{D}_{K_{q,P^\alpha }/\mathbb{F}_q(T)})) 
\\ &= \mathfrak{D}^{-1}_{K_{q^d,P^\alpha }/\mathbb{F}_{q^d}(T)}\text{con}_{\mathbb{F}_{q^d} K_{q,P^\alpha }/K_{q^d,P^\alpha }}( \text{con}_{K_{q,P^\alpha }/\mathbb{F}_{q^d} K_{q,P^\alpha }}(\mathfrak{D}_{K_{q,P^\alpha }/\mathbb{F}_q(T)})) \\\notag &= \mathfrak{D}^{-1}_{K_{q^d,P^\alpha }/\mathbb{F}_{q^d}(T)} \text{con}_{K_{q,P^\alpha }/K_{q^d,P^\alpha }}(\mathfrak{D}_{K_{q,P^\alpha }/\mathbb{F}_q(T)}) \\\notag & = \left[\prod_{i=1}^s  \prod_{\mathfrak{P}|\wp_i} \mathfrak{P} \right]^{-({\alpha q^{d\alpha }- (\alpha +1)q^{d(\alpha -1)}})}  \prod_{\mathfrak{B}|\mathfrak{p}_\infty} \mathfrak{B}^{-(q^d-2)} \\\notag & \qquad \qquad\qquad\qquad\qquad  \times \text{con}_{K_{q,P^\alpha }/K_{q^d,P^\alpha }}\left(\mathfrak{P}^{{\alpha q^{s\alpha } - (\alpha +1)q^{s(\alpha -1)}}} \prod_{\mathfrak{B}|\mathfrak{p}_\infty} \mathfrak{B}^{q-2}\right) \\\notag &= \left[\prod_{i=1}^s  \prod_{\mathfrak{P}|\wp_i} \mathfrak{P} \right]^{-({\alpha q^{d\alpha }- (\alpha +1)q^{d(\alpha -1)}})}  \prod_{\mathfrak{B}|\mathfrak{p}_\infty} \mathfrak{B}^{-(q^d-2)}  \\\notag & \qquad\qquad\qquad  \times \left[\prod_{i=1}^s  \prod_{\mathfrak{P}|\wp_i} \mathfrak{P}^{{\alpha q^{s\alpha } - (\alpha +1)q^{s(\alpha -1)}}}\right]^{\frac{q^{d\alpha } - q^{d(\alpha -1)}}{q^{\alpha s}-q^{(\alpha -1)s}}} \left[\prod_{\mathfrak{B}|\mathfrak{p}_\infty} \mathfrak{B}^{q-2}\right]^{\frac{q^d-1}{q-1}}.\end{align*} 
The result follows.  	
 \end{proof}

 \section{Tower structures}

\label{WTCM}
We now examine cyclotomic function fields as composites of towers of Kummer and Artin-Schreier extensions. Let $P \in \mathbb{F}_q[T]$ be irreducible, whence $P$ possesses only simple roots. We assume once more that $\deg{P}=s \mid d$. We now give the recursive definition for the cyclotomic function field $K_{q^d,(T-\rho)^\alpha}$. 
 
 \begin{lemma} Let $(T - \rho) \in A_d$ be a factor of $P$, where $\rho\in \mathbb{F}_{q^s}\cong A_1/P$. Then the field $K_{q^d,(T - \rho)^\alpha}$ may be described recursively by a tower of composita of explicitly determined Kummer and Artin-Schreier extensions over $\mathbb{F}_{q^d}(T)$. \end{lemma}
\begin{proof}  We first consider the case $\alpha = 2$, i.e., the cyclotomic function field $K_{q^d,(T-\rho)^2}$ generated over $\mathbb{F}_{q^d}(T)$ by the torsion points of $(T-\rho)^2$. By definition of the Carlitz action $*_d$, we have 
\begin{align*}
(T-\rho)*_d u &=  u^{q^d}+(T-\rho) u \\
(T-\rho)^2*_du & =  u^{q^{2d}}+ 
\left( 
(T-\rho)^{q^d}+(T-\rho)
\right)	u^{q^d}
+
(T-\rho)^2 u.
\end{align*}
We denote $X=(T-\rho)*_du$. The equation 
\[
(T-\rho)*_dX=X^{q^d}+(T-\rho)X=0
\]
implies that either $X=0$ or $X^{q^d-1}=-(T-\rho)$. Let $\lambda$ be a solution of the second equation, so that $\lambda^{q^d-1}=-(T-\rho)$. The general torsion point for $(T-\rho)^2$ over $\mathbb{F}_{q^d}$ may then be described by the equation 
\[
u^{q^d}+ (T-\rho)u =\lambda.
\] 
We set $u=U\lambda$, which yields
\[
U^{q^d}\lambda^{q^d}-\lambda^{q^d-1} \lambda U=\lambda,
\]
whence we obtain
\begin{equation} \label{U-def}
U^{q^d}-U=-\frac{1}{T-\rho}.
\end{equation}
We have thus constructed the following tower of fields:
\[
\xymatrix{
 & K_{q^d,(T-\rho)^2} \ar@{-}[dr] \ar@{-}[dl] & \\
 \frac{\mathbb{F}_{q^d}(T)[U]}{\langle U^{q^d}-U=-\frac{1}{T-\rho} \rangle} \ar@{-}[dr] &  &
 \frac{\mathbb{F}_{q^d}(T)[\lambda]}{\langle\lambda^{q^d-1}=-(T-\rho)\rangle} \ar@{-}[dl] \\
 & \mathbb{F}_{q^d}(T) &
}
\]
The field $K_{q^d,(T-\rho)^2}$ is therefore the compositum of a Kummer extension and an Artin-Schreier extension, where a root $u$ of the torsion point equation $(T-\rho)^2*_du$ is given by $u=U\lambda$. One may now easily proceed inductively: Let $U_2:=U$ be the element given in \eqref{U-def}, and let $u_2=U_2\lambda$. A solution of $(T-\rho)^3*_du_3=0$ is then given by 
\[
u_3^{q^d}+(T-\rho)u_3=u_2.
\]
We set $U_3=u_3/\lambda$. This yields $U_3^{q^d}\lambda^{q^d}- \lambda^{q^d-1}\lambda U_3=u_2$, which in turn implies that
\[
U_3^{q^d}-U_3=-{ \frac{U_2}{T-\rho} }. 
\]
(We note that $u_2/\lambda^{q^d-1}=U_2$.) In this way, we may build a tower of successive Artin-Schreier extensions. We have thus obtained the extended diagram 
\begin{equation} \label{compositediagram}
\xymatrix{
  & K_{q^d,(T-\rho)^\alpha} \ar@{-}[dr] \ar@{-}[ddddl] \ar@{.}[dd]& 
  \\
  &   &
  \frac{\mathbb{F}_{q^d}(T)[U_\alpha]}{\langle U_\alpha^{q^d}-U_\alpha={-\frac{U_{\alpha-1}}{T - \rho}} \rangle} \ar@{.}[dd]\\
  & 
K_{q^d,(T-\rho)^3} \ar@{-}[dr] \ar@{-}[ddl] \ar@{-}[d]
   &  \\
  &  
K_{q^d,(T-\rho)^2} \ar@{-}[dr] \ar@{-}[dl] \ar@{-}[d]
  & 
  \frac{\mathbb{F}_{q^d}(T)[U_3]}{\langle U_3^{q^d}-U_3={-\frac{U_2}{T-\rho}} \rangle} \ar@{-}[d]
  \\
 \mathbb{F}_{q^{d}}(T)(\lambda) \ar@{-}[dr]_{\text{Kummer}}  & 
K_{q^d,(T-\rho)} \ar@{=}[l] 
  & 
 \frac{\mathbb{F}_{q^d}(T)[U_2]}{\langle U_2^{q^d}-U_2={-\frac{1}{T-\rho}}\rangle} \ar@{-}[dl]\\   
	 & \mathbb{F}_{q^d}(T) & 
}
\end{equation}
This completes the proof. \end{proof}

Within the diagram \eqref{compositediagram}, we have described the Kummer model of the extension $\mathbb{F}_{q^d}K_{q,P^\alpha}/\mathbb{F}_{q^d}(T)$. By the arguments of \S 2 on Kummer covers, it thus remains to describe the Artin-Schreier-Witt model of the extension $$K_{q^d,P^\alpha}^{\mathbb{F}_{q^d}^*}/\mathbb{F}_{q^d}(T).$$ We may therefore prove the following corollary.
\begin{corollary} Let $P \in \mathbb{F}_q[T]$ be irreducible of degree $s\mid d$. Then the field $K_{q^d,P^\alpha}^{\mathbb{F}_{q^d}^*}$ may be described recursively by a tower of composita of explicitly determined Artin-Schreier extensions over $\mathbb{F}_{q^d}(T)$. \end{corollary} 
\begin{proof} For each linear factor $T-\rho_i$ of $P$, we consider the generating elements
$U_j^{(i)}$, $j = 1,\ldots,\alpha$, of the fields $K_{q^d,(T-\rho_i)^\alpha}$, each of which satisfies the equation
\[
\left( U_j^{(i)} \right)^p- U_j^{(i)} = U_{j-1}^{(i)}.
\] 
By the Chinese remainder theorem once more, we obtain an equality with the compositum
\[
K_{q^d,P^\alpha}^{\mathbb{F}_{q^d}^*} = \prod_{i=1}^{{s}} K_{q^d,(T-\rho_i)^\alpha}^{\mathbb{F}_{q^d}^*}.
\]
As before, let $\sigma$ be a generator of the cyclic group $\mathrm{Gal}(\mathbb{F}_{q^d}/\mathbb{F}_{q})$ such that $\sigma (\rho_i)=\rho_{i+1}$, for each $i=1,\ldots,{{s}}-1$, and $\sigma(\rho_{{s}})=\rho_1$. For each $i=1,\ldots,{{s}}-1$, we have $$\sigma(U_2^{(i)})=U_2^{(i+1)}-a_i, \qquad \qquad a_i\in \mathbb{F}_{q^d}.$$ We may normalise without loss the selection of elements $U_2^{(i)}$ so that 
$$\sigma(U_2^{(i)})=U_2^{(i+1)}.$$ Proceeding inductively, we thus obtain
\[
\sigma(U_{j}^{(i)})=U_{j}^{(i+1)} \qquad\qquad  j = 2,\ldots,\alpha.
\]
By the correspondence of Artin-Schreier extensions to additive characters, we know that the field $\mathbb{F}_{q^d}K_{q,P^2}^{\mathbb{F}_{q}^*}$ is given by an Artin-Schreier equation
\[
y^{q^d}-y=\sum_{i=1}^{{s}} \frac{b_i}{U_1^{(i)}}, \qquad \qquad U_1^{(i)}=T-\rho_i.
\]
On the other hand, the field $K_{q,P^{2}}$ is invariant under the action of the generator $\sigma$. It follows that there exists an element $c_i \in \mathbb{F}_{q^d}(T)$ such that 
\[
\sigma\left(\sum_{i=1}^{{s}} \frac{b_i}{U_1^{(i)}}\right)=
\sum_{i=1}^{{s}} \frac{b_i}{U_1^{(i+1)}}+ c_i^p-c_i.
\]
By uniqueness of partial fraction expansions, this yields that $b_1=b_2=\cdots=b_{{s}}$. In order to consider the higher powers $j = 3,\ldots,\alpha$, one may then proceed inductively, or alternatively, via a standard Witt vector construction \cite{Hazewinkel}. \end{proof}

\section{Galois module structure}

The space of holomorphic differentials $H^0(X,\Omega_X)$ for the curve $X$ corresponding to the cyclotomic function field $K_{q^d,M}$  if $\mathbb{F}_{q^d}$  is selected big enough so that $M$ splits in $\mathbb{F}_{q^d}$ are known, see \cite{Ward}. Our strategy in computing holomorphic differentials $H^0(Y,\Omega_Y)$ for the curve $Y$ corresponding to $K_{q,M}$, where $M$ does not split in $\mathbb{F}_q$ is to consider the Galois cover $X\rightarrow Y$ with Galois group $H=H_{q,M}$ and reduce holomorphic differentials of $X$ to holomorphic differentials of $Y$. We will prove that we have the inclusions
\begin{equation}
\label{inclusions}
H^0(Y,\Omega_Y) \subset L_Y(\Omega(D)) =H^0(X,\Omega_X)^H,
\end{equation}
where $L_Y(\Omega(D))$ is a space of non-holomorphic differentials allowing poles on a certain explicitly given divisor $D$. Our proposed strategy is to  approach the structure of $L_Y(\Omega(D))$ using the inclusions of eq. (\ref{inclusions}. We have to compute the  $H$-invariant differentials of $H^0(X,\Omega_X)$ and then select the holomorphic ones among them. 

Let us start in a more general setting. 
Consider a Galois ramified cover  $\pi:X\rightarrow Y$  of projective complete nonsingular curves defined over the field $k=\mathbb{F}_{q^d}$ with Galois group $H$, and let $df$ be a differential on $Y$. Let $\pi^*$ denote the pullback and $k(X),k(Y)$ be the functions fields of the curves $X,Y$. 
We will follow a multiplicative notation for the divisors. A divisor $D=\prod_{i=1}^t P_i^{a_i}$, 
where $P_i$ are prime divisors and
 $a_i\in \mathbb{N}$ will be called integral.
The divisor $\mathrm{div}_X(\pi^*(df))$ of $\pi^*(df)$ in $X$ is given by 
\[
\mathrm{div}_X(\pi^*(df))=
\pi^* \mathrm{div}_Y(df)\cdot R_{X/Y},
\] 
where $R_{X/Y}$ is the ramification divisor, see \cite[IV.2]{Hartshorne:77}. This allows us to compute holomorphic differentials on $X$ via functions $g\in k(X)$ such that $ \mathrm{div}_X(g)\cdot\mathrm{div}_X(\pi^*df)$ is holomorphic \cite[sec. 3]{Boseck}. $H$-invariant differentials on $X$ correspond to meromorphic differentials $\Omega_Y(\mathcal{D})$ of $Y$ with a well-prescribed set of poles and pole orders, where $\mathcal{D}$ is a divisor which can be explicitly prescribed. 

We let $\mathfrak{I}_X$ (resp. $\mathfrak{I}_Y$) denote the collection of integral divisors of $X$ (resp. $Y$).
Let $df$ be a  differential of $Y$, and let 
\[
R_{X/Y}=\prod_{i=1}^s  \prod_{Q_{\nu,i} \mapsto P_i}  Q_{\nu,i}^{\delta_{i}}
\]
be the ramification divisor of $X/Y$, 
where $\delta_i \in \mathbb{Z}$ is the exponent of the different. 
\begin{lemma}
The space of $H$-invariant differentials in $X$ is isomorphic to the vector space 
\begin{align}
\label{H-inv}
L_Y\left( 
\Omega_Y
\left(\prod_{i=1}^s  (P_i)_Y^{\lf \frac{\delta_i}{e_i}\rf } \right)\right)
&= \left\{ g\in k(Y): 
\mathrm{div}_Y(g) \cdot \mathrm{div}_Y({df}) \prod_{i=1}^s  (P_i)_Y^{\lf \frac{\delta_i}{e_i}\rf } \in \mathfrak{I}_Y
\right\}
\\
& = L_Y(\mathcal{D}) \label{invariants1}
\end{align}
\end{lemma}
\begin{proof}
 An $H$-invariant  differential on $X$ is given by 
$h df, $ where $h,f \in k(Y)$. We observe that $\omega=h df$ is holomorphic if, and only if, the divisor
 \[\mathrm{div}_X(\omega)=\mathrm{div}_X(h)\mathrm{div}_X(\pi^*(df)) \in \mathfrak{I}_X.
 \] 
 By   taking the pushforward and using the fact that applying $\pi_* \pi^*$ is equivalent to raising to the  $|H|$ power
 we compute that 
\[
\pi_* (\mathrm{div}_X(\omega ))=
\mathrm{div}_Y(h)^{|H|}\cdot \mathrm{div}_Y(df)^{|H|}
 \cdot \mathrm{div}_Y(\pi_*(R_{X/Y})),
\]
and $\omega$ is holomorphic if and only if $\pi_*(\omega) \in \mathfrak{I}_Y$.
An easy computation yields that 
\[
\mathrm{div}_Y(\pi_*(R_{X/Y}))=\prod_{i=1}^s  (P_i)_Y^{\delta_i \frac{|H|}{e_i}},
\]
where $\delta_i$ is the differential exponent at $P_i$, and therefore, the set of $H$-invariant differentials can be identified with the space of functions given in eq. (\ref{H-inv}).
\end{proof}
Assume now that  $X$ and $Y$ are the curves corresponding to the function fields 
	$K_{q^d,P^\alpha}$ and $\mathbb{F}_{q^d} K_{q,P^\alpha}$, respectively. 
We will take $f=T$ and we will compute the ingredients of lemma \ref{H-inv}.

We will compute the divisor in $Y$ of $dT$. 
We know that 
\begin{equation}
\label{divisorT}
\mathrm{div}_Y(dT)= 
\prod_{i=1}^{s} 
\prod_{\mathfrak{P}|P_i} 
{\mathfrak{P}_i}^S \prod_{Q \mid \mathfrak{p}_\infty} 
Q^{q-2} \cdot \text{con}_{\mathbb{F}_q(T)/\mathbb{F}_{q^d}K_{q,P^\alpha}}(\mathfrak{p}_\infty^{-2}),
\end{equation}
where $S=\alpha q^\alpha-(\alpha+1)q^{\alpha-1}$ \cite[prop. 12.7.2]{salvador2007topics}. 
Recall also the value of  the different divisor we have computed in proposition \ref{prop:47}
$\mathfrak{D} = \mathfrak{D}_{K_{q^d,P^\alpha}/\mathbb{F}_{q^d} K_{q,P^\alpha}} $ 
$$\mathfrak{D}_{K_{q^d,P^\alpha}/\mathbb{F}_{q^d} K_{q,P^\alpha}} = \left(\prod_{i=1}^s  \prod_{\mathfrak{P}|\wp_i} \mathfrak{P}\right)^{\frac{q^{d (\alpha-1)} \left(q^d-q^s\right)}{q^s-1}}  \left(\prod_{\mathfrak{B}|\mathfrak{p}_\infty} \mathfrak{B}\right)^{q\left(\frac{q^{d-1} - 1}{q-1}\right)}.$$
\begin{remark}
In the special case that $s=d$, ramification is tame in the extension $K_{q^d,P^\alpha}/\mathbb{F}_{q^d}K_{q,P^\alpha}$, so that $\mathcal{D}=0$ and the invariant elements thus satisfy $L_X(\Omega_X)^H=L_Y(\Omega_X)$.
\end{remark}

By the computation of the second author in \cite[p.46]{Ward}
A holomorphic differential on $X$ is given by a differential of the form 
\begin{equation} \label{hol13}
\omega= \prod_{i=1}^{{s}} \prod_{k=2}^\alpha \lambda_{i,k}^{\mu_{i,k}} \lambda_{i,1}^{-\mu_{i,1}} dT, 
\end{equation}
where $\lambda_{i,k}$ are  generators of the Carlitz torsion modules $C_{q^d}[P_i^k]$ of the factors of $M=\prod_{i=1}^r P_i^{n_i}$, and $\mu_{i,k}$ satisfy certain inequalities, see \cite[eq. (22)]{Ward}.
By the Chinese remainder theorem, the class of an element $D  \in \mathbb{F}_{q^d}[T]$ modulo $$P^\alpha = \prod_{i=1}^s P_i^\alpha$$ determined by the classes $D \mod P_i^\alpha$. Moreover, a class $D \mod P_i^\alpha$ can be expressed as a $P_i$-adic series
\[
D \mod P_i^\alpha = a_{i,0} +a_{i,1} P_i + a_{i,2} P_i^2 + \cdots a_{i,\alpha-1} P_i^{\alpha-1} \mod P_i^\alpha,
\]
where $a_{i,j} \in \mathbb{F}_{q^d}$ for all $1\leq i \leq s$ and $0 \leq j \leq \alpha-1$. The action of $D$ on $\lambda_{i,k}$ is determined by its $P_i$-adic decomposition.
We also have for each $k=1,\ldots,\alpha$ that  
\begin{align}
\left(\sum_{\ell=0}^{\alpha-1} a_{i,\ell} P_i^\ell\right)*_d \lambda_{i,k} &= 
\left(\sum_{\ell=0}^{\alpha-1} a_{i,\ell} P_i^\ell P_i^{\alpha-k} \right) *_d \lambda_{i,\alpha} \nonumber \\
&=   
\left(\sum_{\ell=0}^{\alpha-1} a_{i,\ell}  P_i^{\alpha- k + \ell} \right) *_d \lambda_{i,\alpha} \label{actionl} \\
&=   
\sum_{\ell=0}^{\alpha-1} a_{i,\ell}  \left(P_i^{\alpha- k + \ell} *_d \lambda_{i,\alpha}\right) \nonumber \\
&=  
\sum_{\ell=0}^{\alpha-1} a_{i,\ell} \lambda_{i,k-\ell},\nonumber \end{align}
where we define $\lambda_{i,m} := 0$ whenever $m \leq 0$. 

The action described in \eqref{actionl} gives rise to a representation 

\begin{eqnarray}\notag
\rho:(\mathbb{F}_{q^d}(T)/P^\alpha)^* & \longrightarrow & \mathrm{GL}(\alpha,\mathbb{F}_{q^d}) \\
\sum_{\ell=0}^{\alpha-1} a_\ell P^\ell 
& \mapsto & 
\begin{pmatrix}  \label{actionMatrixav}
a_0 & a_1 & \cdots & a_{\alpha-2} & a_{\alpha-1} \\
0   & a_0 & a_1  & \cdots   & a_{\alpha-2}  \\
\vdots & 0 & \ddots & \ddots &  \vdots\\
\vdots &   & \ddots & a_0 & a_1\\
0 & \cdots & \cdots & 0 & a_0
\end{pmatrix}.  
\end{eqnarray}
Upper triangular matrices of the form \eqref{actionMatrixav} are known as {\em Toeplitz} matrices. In this way, we obtain a representation of $\mathbb{F}_{q^d}^*$ inside the 
algebraic subgroup  of upper triangular matrices $T_\alpha(\mathbb{F}_{q^d})$ given by the ideal of relations 
\[
I:=\langle T_{i,i+\mu}-T_{j,j+\mu} \text{ for all } 1\leq i \leq \alpha 
\text{ and }  0\leq j \leq n-1\rangle,  
\]
i.e., the algebraic group given by the affine coordinate ring
\[
k[T_{i,j},\det(T_{i,j})^{-1}]/\langle T_{i,j}: i> j, I\rangle
\] and the corresponding finite group of Lie type defined by the fixed points of the $d$-th power of the Frobenius.  
\begin{definition} Let $\Delta T_\alpha$ denote the subgroup of $T_\alpha(\mathbb{F}_{q^d})$ consisting of matrices given as in \eqref{actionMatrixav}.
\end{definition}

\begin{remark}\label{select}
Let $M \in \mathbb{F}_q[T]$ be arbitrary and nonzero. An element $f$ is invariant under the action of $H_{q^d,M}$ if, and only if,
\begin{equation} \label{comp-cond}
\boxed{
\sigma_D \circ \sigma(x)=\sigma \circ\sigma_D(x), \text{ for all } D\in \mathbb{F}_{q^d}[T]/M.
}
\end{equation}
We now recall a classical result of descent theory. From a cohomological point of view, there is a natural action of $\mathrm{Gal}(\mathbb{F}_{q^d}/\mathbb{F}_q)$ on the $\mathbb{F}_{q^d}$-vector space $$V=H^0(K_{q^d,M},\Omega_{K_{q^d,M}}).$$
Let us consider a basis $\{\omega_1,\ldots,\omega_g\}$ of $V$, where $g$ denotes the genus of $K_{q^d,M}$. Define the map 
\begin{eqnarray*}
\mathrm{Gal}(\mathbb{F}_{q^d}/\mathbb{F}_q) &\rightarrow & \mathrm{GL}(V) \\
\sigma &\mapsto & \rho(\sigma): v \mapsto v^{\sigma}
\end{eqnarray*}
For each basis element $\omega_i$, $1\leq i \leq g$ we write
\[
\omega_i^\sigma= \sum_{\nu=1}^g \rho(\sigma)_{\nu,i} \omega_\nu.
\]
Then, since $(\omega_i^{\sigma_1})^{\sigma_2}=\omega^{\sigma_1\sigma_2}$, we have
\[
\omega_i^{\sigma_1 \sigma_2}=\sum_{\nu,\mu=1}^{g}
\rho(\sigma_1)^{\sigma_2}_{\nu,i} \rho(\sigma_2)_{\mu,\nu} \omega_\mu
\]
so that the function $\rho$ satisfies the cocycle condition $\rho(\sigma_1\sigma_2)=\rho(\sigma_1)^{\sigma_2} \rho(\sigma_2).$
The multidimensional Hilbert's 90 theorem asserts that  there is an element $\wp \in GL(V)$ such that $\rho(\sigma)=\wp^{-1}\wp^\sigma$. Moreover, the elements $\omega_i'=\omega_i \wp^{-1}$ are $\mathrm{Gal}(\mathbb{F}_{q^d}/\mathbb{F}_{q})$-invariant since 
\[
(\omega_i \wp^{-1})^\sigma=(\omega_i^\sigma)(\wp^{-1})^\sigma=
\omega_i \rho(\sigma)( \wp^{-1})^\sigma=\omega_i \wp^{-1}\wp^\sigma(\wp^{-1})^\sigma=\omega_i \wp^{-1}.
\]
From now on $\{\omega_i\}_{i=1,\ldots,g}$  denotes an $\mathrm{Gal}(\mathbb{F}_{q^d}/\mathbb{F}_{q})$-invariant basis. 
An arbitrary element $\omega \in V $ is written as
\[
\omega= \sum_{i=1}^g a_i \omega_i
\]
and the action of $\sigma_D$ with respect to this basis is given by a representation
\begin{eqnarray*}
G_{q^d,M}  & \rightarrow & GL(V) \\
 \sigma_D & \mapsto & A(\sigma_D)=:A_D
\end{eqnarray*} The condition of \eqref{comp-cond} is now expressed as the matrix condition 
\[
A_D 
\begin{pmatrix}
a_1^q \\ \vdots \\a_g^q
\end{pmatrix}
=A_D^q
\begin{pmatrix}
a_1^q \\ \vdots \\a_g^q
\end{pmatrix}
\Leftrightarrow
A_D^{q^{d-1}}
\begin{pmatrix}
a_1 \\ \vdots \\a_g
\end{pmatrix}
=A_D
\begin{pmatrix}
a_1 \\ \vdots \\a_g
\end{pmatrix}.
\]
Equivalently, the common eigenspace in $V$ of the eigenvalue 1 of all matrices of the form $A_D^{q^{d-1}-1}$ gives the space of $H_{q^d,M}$-invariant differentials. The space $L_Y(\Omega_Y(\mathcal{D}))$, see \eqref{invariants1} consists of meromorphic differentials with allowed poles at $\mathcal{D}$, and it is only necessary to select the holomorphic differentials among these. \end{remark}

We now return to the case $M = P^\alpha$ and examine invariant rings. The group $\Delta T_\alpha$ is clearly abelian. The action of \eqref{actionMatrixav} has an extension to the polynomial ring $\mathbb{F}_{q^d}[\lambda_{i,0},\ldots,\lambda_{i,\alpha-1}]$. Consider the vector space of polynomials of multidegree, which we denote by $\mathrm{mdeg}$, bounded by $(\mu_0,\ldots,\mu_{\alpha-1})$:
\[
W_{\bar{\mu}}:=\{ f \in \mathbb{F}_{q^d}[\lambda_0,\ldots,\lambda_{\alpha-1}]
: \mathrm{mdeg}(f) \leq (\mu_0,\ldots,\mu_{\alpha-1})
\} \qquad (\bar{\mu} = (\mu_0,\ldots,\mu_{\alpha-1})).
\]
The space $W_{\bar{\mu}}$ inherits a unique representation of the group $(\mathbb{F}_{q^d}(T)/P^\alpha)^*$. Therefore, we obtain 
\begin{equation}
\label{secondInvariant}
H^0(Y,\Omega_Y) \subset L_Y(\Omega_Y(\mathcal{D}))= W_{\bar{\mu}}^{H_{q^d,P^\alpha}}.
\end{equation}
We thus consider the polynomial ring $\mathbb{F}_{q^d}[\lambda_0,\ldots,\lambda_{\alpha-1}]$ equipped with the natural extension of the linear action of \eqref{actionl} which is represented by the matrix in \eqref{actionMatrixav}. The space of invariants $\mathbb{F}_{q^d}[\lambda_0,\ldots,\lambda_{\alpha-1}]^{H_{q^d,P^\alpha}}$
is a finitely generated ring. 

Consider the subgroup $H < \mathrm{GL}(\alpha,\mathbb{F}_{q^d})$ which consists of unitriangular elements of the form \eqref{actionMatrixav}, corresponding to the wild component of the cover $X \rightarrow Y$. Let us view the particular case of $\mathrm{GL}(3,\mathbb{F}_3)$, in order to simplify the presentation. In this case, the subgroup $H$ of Toeplitz matrices take the form 
\[ \sigma = \begin{pmatrix}
1 & a & b \\
0 & 1 & a \\
0 & 0 & 1
\end{pmatrix}, \qquad\qquad a,b \in \mathbb{F}_3.\] Viewing this matrix as acting on the ring $\mathbb{F}_3[x,y,z]$, we find $\sigma(x) = x$, and $$\sigma(y^3 - y x^2) = y^3 + a^3 x^2 - y x^2 - a x^3 = y^3 + a x^2 - y x^2 - a x^3 = y^3 - y x^2.$$ By modular representation theory \cite[Corollary 3.1.6]{CampWeh}, a homogeneous system of parameters $\{f,g,h\}$ for $\mathbb{F}_3[x,y,z]$ satisfies $$\mathbb{F}_3[f,g,h] = \mathbb{F}_3[x,y,z]^H$$ if, and only if, $$\deg(f) \deg(g) \deg(h) = |H|.$$ By definition, $|H| = 9$. Also by definition, we have $$\sigma(y) = y + ax,$$ so any invariant involving $y$ must have degree at least $q$. The same is true for $z$ as $$\sigma(z) = z + ay + bx.$$ As the invariant for $x$ is already of minimal degree (one), it follows that the remaining two invariants must be of degree $3 \mid 9$ if $\mathbb{F}_3[x,y,z]^H$ is polynomial. Let us then examine the action of $\sigma$ on all homogeneous terms of degree $3$ in the variables $x$, $y$, and $z$: 
\begin{center}{\tiny
 $$ \begin{array}{||c |c| c| c |c||} 
 \hline
 \text{Element} & \text{Constant} & x & x^2 & x^3 \\ 
 \hline\hline
 x^3 & 0 & 0 & 0& 0 \\ 
 \hline
 y^3 & 0 & 0 & 0 &  a \\
 \hline
 z^3 & a y^3 & 0 & 0 & b  \\
 \hline
 x^2 y & 0 & 0 & 0 & a \\
 \hline
 xy^2 & 0 & 0 & 2ay & a^2  \\ 
 \hline
  x^2 z & 0 & 0 & ay & b \\ 
 \hline
 x z^2 & 0 & a^2 y^2 + 2 a y z + z^2 & 2 a b y + 2 b z & b^2 \\
 \hline
 y^2 z & a y^3 + y^2 z  & 2 a^2 y^2 + b y^2 + 2 a y z  & a^3 y + 2 a b y + a^2 z & a^2 b  \\
 \hline
 y z^2 & a^2 y^3 + 2 a y^2 z + y z^2 & a^3 y^2 + 2 a b y^2 + 2 a^2 y z + 2 b y z + a z^2 & 2 a^2 b y + b^2 y + 2 a b z & a b^2  \\
 \hline
 xyz & 0 & a y^2 + y z & a^2 y + b y + a z & ab \\ 
 \hline
\end{array}$$} \emph{The table of remainder coefficients via the $H$-action.}
\end{center}
Note that we are only considering the remainder coefficients, as the unitriangular action always returns an expansion containing the original element (the left-hand column). Clearly $f=x$ is the first invariant. The invariant of minimal degree containing $y$ is equal to the aforementioned $y^3 - y x^2$. It remains then to find the third invariant of degree $3$, which must contain the variable $z$. Examination of the coefficients in $x^3$, we find that there are only two other terms with equal coefficients: $x^2 z$ and $z^3$. But if an invariant contains these two terms, then as $x^2 z$ has a coefficient of $ay$ on $x^2$, it follows that  it must contain another term with a coefficient equal to a multiply of $ay$ on $x^2$. The occurs in only one element: $x y^2$. But the coefficient of $x^3$ on $x y^2$ is equal to $a^2$, which matches no other terms of degree 3. Thus there is no invariant of degree 3 containing the element $z$. This argument may be easily generalised to our case of $\mathrm{GL}(\alpha,\mathbb{F}_{q^d})$. We thus obtain:

\begin{theorem} Let $H$ be the subgroup of upper unitriangular Toeplitz matrices in $\mathrm{GL}(\alpha,\mathbb{F}_{q^d})$. The ring of invariants $\mathbb{F}_{q^d}[\lambda_0,\ldots,\lambda_{\alpha-1}]^H$ is not polynomial. \end{theorem}

By Lemma \ref{Kummerfun}, we have $H_{q^d,P^\alpha}=(\sigma-1)G_{q^d,P^\alpha}$. When $\alpha = 3$, for example, an element of $H_{q^d,P^\alpha}$ takes the form (up to the tame part, i.e., a diagonal multiple)

\begin{align}\label{sigma-1}\notag
\begin{pmatrix}
1 & \sigma(a) & \sigma(b) \\
0 & 1 & \sigma(a) \\
0 & 0 & 1
\end{pmatrix}
\cdot
\begin{pmatrix}
1 & a & b \\
0 & 1 & a \\
0 & 0 & 1
\end{pmatrix}^{-1}&=\begin{pmatrix}
1 & a^q & b^q\\
0 & 1 & a^q \\
0 & 0 & 1
\end{pmatrix}
\cdot
\begin{pmatrix}
1 & a & b \\
0 & 1 & a \\
0 & 0 & 1
\end{pmatrix}^{-1}\\ & = 
\left(
\begin{array}{ccc}
 1 & a^q-a & a\big(a^q-a\big)+b^q-b \\
 0 & 1 & a^q-a \\
 0 & 0 & 1 \\
\end{array}
\right).
\end{align}
This may be easily generalised to arbitrary $\alpha$:

\begin{proposition} The upper diagonal entries of matrices for elements of $H_{q^d,P^\alpha}$ are $\mathbb{F}_{q^d}$-multiples of Frobenius differences $x^q - x$ $(x \in \mathbb{F}_{q^d})$. \end{proposition}

An explicit description of the holomorphic differentials of $K_{q,P^\alpha}$ requires computation of classical binomial coefficients occurring in the $H_{q^d,P^\alpha}$-action on the canonical basis consisting of elements of the form \eqref{hol13}. Although it would be desirable, no basis of cyclotomic holomorphic differentials is currently known which allows a description of this action in terms of the function field binomial coefficients $\bfrac{M}{i}_{q^d}$ appearing in the expansion of the additive Carlitz action \eqref{Carlitzaction}
 $$C_{q^d}(M)(u) = \sum a_i \tau_{q^d}^i (u) = \sum_{i=0}^{\deg(M)} \bfrac{M}{i}_{q^d} u^{q^{di}}.$$ This fact is owed to linear dependence between terms: The Carlitz $M$-torsion module $C_{q^d}[M]$ is too small as an $\mathbb{F}_{q^d}$-vector space of dimension equal $\deg(M)$, far from sufficient to match genus growth \cite[Proposition 12.7.1]{salvador2007topics}. We leave this as an open question for a future work.

\end{document}